\documentclass[smallcondensed]{svjour3}
\usepackage{amsfonts,amssymb,amsmath}
\usepackage{cite}
\usepackage{color}
\usepackage[all]{xy}
\smartqed
\newcommand{\al}{\alpha}
\newcommand{\be}{\beta}
\newcommand{\ga}{\gamma}

\newcommand{\de}{\delta}
\newcommand{\eps}{\varepsilon}
\newcommand{\bx}{\bar x}
\newcommand{\by}{\bar y}

\newcommand{\iv}{^{-1} }
\newcommand{\Real}{\mathbb R}
\newcommand {\R} {\mathbb R}
\newcommand {\N} {\mathbb N}

\newcommand {\B} {\mathbb B}
\newcommand {\Sp} {\mathbb S}
\newcommand {\gph} {{\rm gph}\,}
\newcommand {\dom} {{\rm dom}\,}

\newcommand {\Er} {{\rm Er}\,}

\newcommand {\sd} {\partial}

\renewcommand{\iff}{$ \Leftrightarrow\ $}
\newcommand{\folgt}{$ \Rightarrow\ $}

\newcommand{\ds}{\displaystyle}
\newcommand{\lsc}{lower semicontinuous}

\newcommand{\RHS}{right-hand side}
\newcommand{\SVM}{set-valued mapping}
\newcommand{\norm}[1]{\left\Vert#1\right\Vert}

\newcommand{\set}[1]{\left\{#1\right\}}

\newcounter{mycount}
\newcommand{\cnta}{\setcounter{mycount}{\value{enumi}}}
\newcommand{\cntb}{\setcounter{enumi}{\value{mycount}}}
\newcommand{\sr}{{}^sr}

\newcommand{\smartqedtr}{\newcommand{\qedtr}{\ifmmode\triangle\else{\unskip\nobreak\hfil
\penalty50\hskip1em\null\nobreak\hfil$\triangle$
\parfillskip=0pt\finalhyphendemerits=0\endgraf}\fi}}
\smartqedtr  
\lccode`\-=`\-

\renewcommand {\theenumi} {\roman{enumi}}

\begin{document}

\title{Error Bounds and H\"older Metric Subregularity
\thanks{The research was supported by the Australian Research Council, project DP110102011.}}

\author{Alexander Y. Kruger}
\institute{A. Y.~Kruger \at
Centre for Informatics and Applied Optimization, Faculty of Science and Technology, Federation University Australia, POB 663, Ballarat, Vic, 3350, Australia\\
\email{a.kruger@federation.edu.au} }

\date{Received: date / Accepted: date}


\dedication{Dedicated to Professor Lionel Thibault}
\maketitle

\begin{abstract}
The H\"older setting of the metric subregularity property of set-valued mappings between general metric or Banach/Asplund spaces is investigated in the framework of the theory of error bounds for extended real-valued functions of two variables.
A classification scheme for the general H\"older metric subregularity criteria is presented.
The criteria are formulated in terms of several kinds of primal and subdifferential slopes.


\end{abstract}

\section{Introduction}

In this article, we investigate the H\"older setting of the metric subregularity property of set-valued mappings between general metric or Banach/Asplund spaces in the framework of the theory of error bounds for extended real-valued functions of two variables developed in \cite{Kru15}.
Several primal and dual space local quantitative and qualitative criteria of H\"older metric subregularity are formulated.
The relationships between the criteria are established and illustrated.

Recall that an extended-real-valued function $f:X\to\Real_\infty:=\R\cup\{+\infty\}$ on a metric space $X$ is said to have a local \emph{error bound} (cf., e.g., \cite{Aze03,AzeCor04,FabHenKruOut10,Iof79}) with constant $\tau>0$ at a point $\bar{x}\in X$ with $f(\bx)=0$ if there exists a neighbourhood $U$ of $\bx$ such that
\begin{equation*}\label{el+}
{\tau}d(x,S(f)) \le f_+(x)\quad \mbox{ for all } x \in U.
\end{equation*}
Here $S(f)$ stands for the lower $0$-level set $\{x\in X\mid f(x)\le0\}$ and $f_+(x):=\max\{f(x),0\}$.

A set-valued mapping
$F:X\rightrightarrows Y$ is a mapping which
assigns to every $x\in X$ a subset (possibly empty) $F(x)$ of $Y$.
We use the notation
$$\gph F:=\{(x,y)\in X\times Y\mid
y\in F(x)\}$$
for the graph of $F$ and $F\iv : Y\rightrightarrows
X$ for the inverse of $F$.
This inverse (which always exists) is defined by
$$F\iv(y) :=\{x\in X\mid y\in F(x)\},\quad y\in Y,$$
and satisfies
$$(x,y)\in\gph F \quad\Leftrightarrow\quad (y,x)\in\gph F\iv.$$

A set-valued mapping $F:X\rightrightarrows Y$ between metric spaces is called (locally) \emph{metrically subregular} (cf., e.g., \cite{Mor06.1,RocWet98,DonRoc09,Pen13}) at a point $(\bx,\by)\in\gph F$ with constant $\tau>0$ if there exist neighbourhoods $U$ of $\bx$ and $V$ of $\by$ such that
\begin{equation*}
\tau d(x,F^{-1}(\by))\le d(\by,F(x)\cap V) \quad \mbox{ for all } x \in U.
\end{equation*}

This property represents a weaker version of the more robust \emph{metric regularity} property which corresponds to replacing $\by$ in the above inequality by an arbitrary (not fixed!) $y\in V$.

It is well known that one can always take $V=Y$ in the definition of metric subregularity, which is, thus, equivalent to the existence of a neighbourhood $U$ of $\bx$
such that
\begin{equation}\label{MR0++}
\tau d(x,F^{-1}(\by))\le d(\by,F(x)) \quad \mbox{ for all } x \in U.
\end{equation}

The metric subregularity (and closely related to it \emph{calmness}) property plays an important role in both theory and applications.
The amount of publications devoted to (mostly sufficient) criteria of metric subregularity is huge.
The interested reader is referred to \cite{HenOut01,HenJou02,HenJouOut02,DonRoc04, ZheNg07,IofOut08,DonRoc09, Lev09,Pen10,ZheNg10,Gfr11,Gfr13,Gfr13.2,ZheOuy11,ZheNg12, ApeDurStr13,NgaiTin}.

In many important situations, the standard metric (sub)regularity property is not satisfied, and more subtle H\"older type estimates come into play.

If instead of \eqref{MR0++} one uses the following more general condition:
\begin{equation}\label{02}
\tau d(x,F^{-1}(\by))\le (d(\by,F(x)))^q\quad \mbox{for all } x\in U,
\end{equation}
where $q\in(0,1]$, then the corresponding property is usually referred to as \emph{H\"older metric subregularity} of order $q$ at $(\bx,\by)$ with constant $\tau$.
The case $q=1$ corresponds to the conventional metric subregularity.
If $q_1<q_2\le1$, then H\"older metric subregularity of order $q_2$ implies that of order $q_2$.

If fixed $\by$ in the above inequality is replaced by an arbitrary $y$ and the inequality is required to hold uniformly over all $y$ near $\by$, then we arrive at the definition of \emph{H\"older metric regularity} of order $q$.
The latter property and even more general nonlinear regularity models have been studied since 1980s; cf. \cite{FraQui12,YenYaoKie08,BorZhu88, Fra87,Fra89,Pen89,Iof13}.

The history of the H\"older metric subregularity property seems to be significantly shorter with most work done in the last few years; cf. \cite{GayGeoJea11,MorOuy,NgaiTronThe3, KlaKruKum12,Kum09,Kla94,LiMor12}.
Note the only attempt so far to consider the case $q>1$ in \cite{MorOuy}.
Nevertheless, the H\"older or more general nonlinear estimates of metric subregularity/error bounds have proved to be important in sensitivity analysis and quantifying linear/sublinear convergence rates for the proximal point and alternating projection
algorithms in optimization and variational inequalities \cite{AttBol09,LiNg09, AttBolRedSou10,LiTanWei10,LiMor12,BorLiYao14,DruIofLew14, NolRon}.

One can easily see that H\"older metric subregularity property \eqref{02} is equivalent to the local error bound property of the extended real-valued function $x\mapsto(d(\by,F(x)))^q$ at $\bx$ (with the same constant).
So one might want to apply to this model the well developed theory of error bounds; cf. \cite{Aze03,AzeCor04,CorMot08,IofOut08, NgaiThe04,NgaiThe08,Pang97, FabHenKruOut10,FabHenKruOut12}.
However, most of the error bound criteria are formulated for \lsc\ functions (cf. Section~\ref{S1}), while the function $x\mapsto(d(\by,F(x)))^q$ can fail to be \lsc\ even when $\gph F$ is closed.

Another helpful observation is that property \eqref{02} can be rewritten equivalently as
\begin{equation*}
\tau d(x,F^{-1}(\by))\le (d(y,\by))^q\quad \mbox{for all } x\in U,\; y\in F(x).
\end{equation*}
This allows one to apply to it the extension of the theory of local error bounds to functions of two variables developed in \cite{Kru15} and used there for characterizing the conventional metric subregularity (case $q=1$).

Following the standard trend initiated by Ioffe \cite{Iof00_}, criteria for error bounds and H\"older metric subregularity of set-valued mappings in metric spaces are formulated in terms of (strong) \emph{slopes} \cite{DMT}.
To simplify the statements in metric and also Banach/Asplund spaces, several other kinds of primal and dual space slopes for real-valued functions and set-valued mappings are discussed in this article and the relationships between them are formulated.
These relationships lead to a simple hierarchy of the error bound and H\"older metric subregularity criteria.

Some statements in the article look rather long because each of them contains an almost complete list of criteria applicable in the situation under consideration.
The reader is not expected to read through the whole list.
Instead, they can select a particular criterion or a group of criteria corresponding to the setting of interest to them (e.g., in metric or Banach/Asplund/smooth spaces, in the convex case, etc.)

The H\"older metric subregularity criteria in terms of primal and dual space limiting derivative-like objects (slopes), while providing a comprehensive theoretical picture of the subject, are not easy to verify in practical problems.
The next natural step (although beyond the scope of the current article) would be to consider classes of mappings corresponding to problems with special structure (semi-algebraic, sub-analytic, etc.) for which computation of the corresponding slopes can be relatively simple.
Such classes can still be sufficiently general to cover important problems arising in applications.
Their regularity properties have attracted recently considerable attention of researchers; cf. \cite{BolDanLew09,Iof09,LiNg09,AttBolSva13,BorLiYao14, DruIofLew14,DruIofLew15,LiMorPha,NolRon}.

The theory of local error bounds of functions of two variables \cite{Kru15}, which is the main tool in the current study, is applicable to more general than H\"older metric subregularity nonlinear models.
This topic is going to be covered in the subsequent article \cite{Kru15.3}.
The author believes that the model developed in \cite{Kru15} can be extended to cover more subtle properties of directional error bounds and directional (H\"older) metric subregularity; cf. \cite{Gfr13,Gfr13.2,NgaiTronThe4,NgaiTronPha}.

Our basic notation is standard, cf. \cite{Mor06.1,RocWet98,DonRoc09,Pen13}.
Depending on the context, $X$ and $Y$ are either metric or normed
spaces.
Metrics in all spaces are denoted by the same symbol
$d(\cdot,\cdot)$;
$d(x,A):=\inf_{a\in{A}}d(x,a)$ is the point-to-set distance from
$x$ to $A$.
$B_\de(x)$ denotes the closed ball with radius $\de$ and centre $x$.
If not specified otherwise, the
product of metric/normed spaces is assumed equipped with the distance/norm given by the maximum of the distances/norms.

If $X$ and $Y$ are normed spaces, their topological duals are denoted $X^*$ and $Y^*$, respectively, while $\langle\cdot,\cdot\rangle$ denotes the bilinear form defining the pairing between the spaces.
The closed unit balls in a normed space and its dual are denoted by $\B$ and $\B^*$, respectively, while $\Sp$ and $\Sp^*$ stand for the unit spheres.

We say that a subset $\Omega$ of a metric space is locally closed near $\bar{x}\in\Omega$ if $\Omega\cap{U}$ is closed for some closed neighbourhood $U$ of $\bar{x}$.
Given an $\al\in\R_\infty$, $\alpha_+$ denotes its ``positive'' part:
$\alpha_+:=\max\{\alpha,0\}$.

If $X$ is a normed linear space, $f:X\to\Real_\infty$, $x\in{X}$, and $f(x)<\infty$, then
\begin{gather}
\partial{f}(x) := \left\{x^\ast\in X^\ast\bigl|\bigr.\;
\liminf\limits_{u\to x,\,u\ne x}
\frac{f(u)-f(x)-\langle{x}^\ast,u-x\rangle}
{\norm{u-x}} \ge 0 \right\}\label{Frsd}
\end{gather}
is the \emph{Fr\'echet subdifferential} of $f$ at $x$.
Similarly, if $x\in\Omega\subset X$, then
\begin{gather}
N_{\Omega}(x) := \left\{x^\ast\in X^\ast\bigl|\bigr.\
\limsup_{u\to x,\,u\in\Omega\setminus\{x\}} \frac {\langle x^\ast,u-x\rangle}
{\|u-x\|} \le 0 \right\}\label{Fr}
\end{gather}
is the \emph{Fr\'echet normal cone} to $\Omega$ at $x$.
In the convex case, sets \eqref{Frsd} and \eqref{Fr} reduce to the subdifferential and normal cone in the sense of convex analysis, respectively.
If $f(x)=\infty$ or $x\notin\Omega$, we set, respectively, $\partial{f}(x)=\emptyset$ or $N_{\Omega}(x)=\emptyset$.
Observe that definitions \eqref{Frsd} and \eqref{Fr} are invariant on the renorming of the space (replacing the norm by an equivalent one).

If $F:X\rightrightarrows Y$ is a set-valued mapping between normed linear spaces and $(x,y)\in\gph F$, then
 \[
D^{*}F(x,y)(y^{*}):=\set{x^{*}\in X^*\mid (x^{*},-y^{*})\in N_{\gph F}(x,y)},\quad y^{*}\in X^*
 \]
is the {\em Fr\'echet coderivative} of $F$ at $(x,y)$.

Several kinds of \emph{subdifferential sum rules} are employed in the article.
Below we provide these rules for completeness.

\begin{lemma}[Subdifferential sum rules] \label{l02}
Suppose $X$ is a normed linear space, $f_1,f_2:X\to\R_\infty$, and $\bx\in\dom f_1\cap\dom f_2$.

{\rm (i) \bf Fuzzy sum rule}. Suppose $X$ is Asplund,
$f_1$ is Lipschitz continuous and
$f_2$
is lower semicontinuous in a neighbourhood of $\bar x$.
Then, for any $\varepsilon>0$, there exist $x_1,x_2\in X$ with $\|x_i-\bar x\|<\varepsilon$, $|f_i(x_i)-f_i(\bar x)|<\varepsilon$ $(i=1,2)$ such that
$$
\partial (f_1+f_2) (\bar x) \subset \partial f_1(x_1) +\partial f_2(x_2) + \varepsilon\B^\ast.
$$

{\rm (ii) \bf Differentiable sum rule}. Suppose
$f_1$ is Fr\'echet differentiable at $\bx$.
Then,
$$
\partial (f_1+f_2) (\bar x) = \nabla f_1(\bx) +\partial f_2(\bx).
$$

{\rm (iii) \bf Convex sum rule}. Suppose
$f_1$ and $f_2$ are convex and $f_1$ is continuous at a point in $\dom f_2$.
Then,
$$
\partial (f_1+f_2) (\bar x) = \sd f_1(\bx) +\partial f_2(\bx).
$$
\end{lemma}

The first sum rule in the lemma above is known as the \emph{fuzzy} or \emph{approximate} sum rule (Fabian \cite{Fab89}; cf., e.g., \cite[Rule~2.2]{Kru03.1}, \cite[Theorem~2.33]{Mor06.1}) for Fr\'echet subdifferentials in Asplund spaces.
The other two are examples of \emph{exact} sum rules.
They are valid in arbitrary normed spaces (or even locally convex spaces in the case of the last rule).
Rule (ii) can be found, e.g., in \cite[Corollary~1.12.2]{Kru03.1} and \cite[Proposition~1.107]{Mor06.1}.
For rule (iii) we refer the readers to \cite[Theorem~0.3.3]{IofTik79} and
\cite[Theorem~2.8.7]{Zal02}.

Recall that a Banach space is \emph{Asplund} if the dual of each its separable subspace is separable; see, e.g., \cite{Mor06.1,BorZhu05} for discussions and characterizations of Asplund spaces.
Note that any \emph{Fr\'echet smooth} space, i.e., a Banach space which admits an equivalent norm, Fr\'echet differentiable at all nonzero points, is
Asplund.
Given a Fr\'echet smooth space, we will always assume that it is endowed with such a norm.

The (normalized) \emph{duality mapping} $J$ between a normed space $Y$ and its dual $Y^*$ is defined as (cf. \cite[Definition~3.2.6]{Lucc06})
\begin{gather}\label{J}
J(y):=\left\{y^*\in \mathbb{S}_{Y^*}\mid \langle y^*,y\rangle=\norm{y}\right\},\quad y\in Y.
\end{gather}
Note that $J(-y)=-J(y)$.
The following simple fact of convex analysis is well known (cf., e.g., \cite[Corollary~2.4.16]{Zal02}).

\begin{lemma}\label{ll01}
Let $(Y,\|\cdot\|)$ be a normed space.
\begin{enumerate}
\item
$\sd\|\cdot\|(y)=J(y)$ for any $y\ne0$.
\item
$\sd\|\cdot\|(0)=\B^*$.
\end{enumerate}
\end{lemma}

The structure of the article is as follows.
In the next section, we present a survey and some extensions of the error bound criteria for a special family of extended-real-valued functions on products of either metric or Banach/Asplund spaces.
The criteria are formulated in terms of several kinds of primal and subdifferential slopes.
The relationships between the slopes are presented.
In Section~\ref{S5}, we demonstrate how H\"older metric subregularity of \SVM s can be treated in the framework of the theory of error bounds.
A collection of slopes for a \SVM\ are discussed.
The final Section~\ref{Criteria} is dedicated to primal and subdifferential criteria for H\"older metric subregularity.

\section{Error bounds and slopes}\label{S1}
In this section, we discuss a general model of error bounds for an extended-real-valued function $f:X\times Y\to\R_\infty$ on a product of metric spaces considered in \cite{Kru15} and recall several error bound criteria in terms of (several kinds of) slopes.

It is assumed that $f(\bx,\by)=0$, and $f$ depends on its second variable in a special way:
\begin{itemize}
\item [(P1)]
$f(x,y)>0$ if $y\ne\by$,
\item [(P2)]
$\ds\liminf_{f(x,y)\downarrow0}\frac{f(x,y)}{d(y,\by)}>0$.
\end{itemize}
In particular, condition $f(x,y)\downarrow0$ implies that $y\to\by$.

Function $f$ is said to have an \emph{error bound} with respect to $x$ at $(\bx,\by)$ with constant $\tau>0$ if there exists a neighbourhood $U$ of $\bx$ such that
\begin{equation}\label{eb}
\tau d(x,S(f)) \le f_+(x,y)\quad \mbox{for all } x\in U,\; y\in Y,
\end{equation}
where
$$S(f):=\{x\in X\mid f(x,\by)\le0\}=\{x\in X|\ f(x,y)\le0 \mbox{ for some } y\in Y\}.$$
The second equality is a consequence of (P1).

The \emph{error bound modulus}
\begin{gather}\label{rr2}
\Er f(\bx,\by):=
\liminf_{\substack{x\to\bx\\f(x,y)>0}} \frac{f(x,y)}{d(x,S(f))}
\end{gather}
coincides with the exact upper bound of all $\tau>0$ such that \eqref{eb} holds true for some neighbourhood $U$ of $\bx$ and provides a quantitative characterization of the \emph{error bound property}.

Observe that the case of a function $f:X\to\R_\infty$ of a single variable can be covered by considering its extension $\tilde f:X\times Y\to\R_\infty$ defined, for some $\by\in Y$, by
\begin{equation*}\label{fti}
\tilde f(x,y)=
\begin{cases}
f(x) &\mbox{if } y=\by,
\\
\infty &\mbox{otherwise}.
\end{cases}
\end{equation*}
Conditions (P1) and (P2) are obviously satisfied.

Definition \eqref{rr2} admits the following equivalent representations.
\begin{proposition}\label{pr-f}
$\ds\Er f(\bx,\by)=
\liminf_{\substack{x\to\bx,\,y\to\by\\f(x,y)>0}} \frac{f(x,y)}{d(x,S(f))}=
\liminf_{\substack{x\to\bx,\,f(x,y)\downarrow0}} \frac{f(x,y)}{d(x,S(f))}$.
\end{proposition}

The roles of variables $x$ and $y$ in definitions \eqref{eb} and \eqref{rr2} are different.
To better reflect this,
we consider the following asymmetric maximum-type distance in $X\times Y$ depending on a positive parameter $\rho$:
\begin{equation}\label{drho}
d_\rho((x,y),(u,v)):=\max\{d(x,u),\rho d(y,v)\}.
\end{equation}
Alternatively, one can use the parametric sum-type distance:
\begin{gather}\label{drho1}
d_\rho^1((x,y),(u,v)):=d(x,u)+\rho d(y,v).
\end{gather}

To formulate (nonlocal) primal space characterizations of the error bound property \eqref{eb}, we are going to use the following two (possibly infinite) constants:
\begin{gather}\label{nls-f}
|\nabla{f}|_{\rho}^{\diamond}(x,y):=
\sup_{(u,v)\ne(x,y)}
\frac{[f(x,y)-f_+(u,v)]_+}{d_\rho((x,y),(u,v))},
\\\label{uss-f}
\overline{|\nabla{f}|}{}^{\diamond}(\bar{x},\by):=
\lim_{\rho\downarrow0}
\inf_{\substack{d(x,\bx)<\rho,\,0<f(x,y)<\rho}}\,
|\nabla{f}|{}^{\diamond}_{\rho}(x,y),
\end{gather}
which are called, respectively,
the \emph{nonlocal $\rho$-slope} of $f$ at $(x,y)$
and
the \emph{uniform strict outer slope}.
It is assumed in \eqref{nls-f} that
$f(x,y)<\infty$.
If $f(x,y)=\infty$, we set $|\nabla{f}|_{\rho}^{\diamond}(x,y)=\infty$.
In \eqref{uss-f}, the usual convention that the infimum of the empty set equals $+\infty$ is in force.

The word ``strict'' reflects the fact that $\rho$-slopes at nearby points contribute to definition \eqref{uss-f} making it an analogue of the strict derivative.
The word ``outer'' is used to emphasize that
only points outside the set $S(f)$ are taken into account.
The word ``uniform'' emphasizes the nonlocal (non-limiting) character of $|\nabla{f}|_\rho^{\diamond}(x,y)$ involved in definition \eqref{uss-f}.
Note that restricting definition \eqref{uss-f} as well as many other definitions of this kind in the rest of the article to only ``outer'' points, while reasonable for characterizing the error bound and (H\"older) metric subregularity properties, prevents one from developing a reasonable calculus of the corresponding derivative-like objects.
If this restriction in the corresponding definitions is dropped, one can still use the resulting objects for formulating certain (stronger!) sufficient conditions. 

In the sequel, superscript `$\diamond$' (diamond) will be used in all constructions derived from \eqref{nls-f} and its analogues to distinguish them from ``conventional'' (local) definitions.

Definition \eqref{nls-f} is a realization of the \emph{nonlocal slope} \cite{FabHenKruOut12} for the case of a function on a product space with the product metric defined by \eqref{drho}.
In definition \eqref{uss-f}, we have not only $x\to\bx$ and $f(x,y)\downarrow0$, but also the metric on $X\times Y$ used in the definition of the nonlocal $\rho$-slope $|\nabla{f}|{}^{\diamond}_{\rho}(x,y)$ changing with the contribution of the $y$ component diminishing as $\rho\downarrow0$.

The local analogues of \eqref{nls-f} and \eqref{uss-f} are defined as follows:
\begin{gather}\label{ls-f}
|\nabla{f}|_{\rho}(x,y):=
\limsup_{\substack{u\to x,\,v\to y\\
(u,v)\ne(x,y)}}
\frac{[f(x,y)-f(u,v)]_+}{d_\rho((u,v),(x,y))},
\\\label{ss-f}
\overline{|\nabla{f}|}{}^{>}(\bar{x},\by):=
\lim_{\rho\downarrow0}
\inf_{\substack{d(x,\bx)<\rho,\,0<f(x,y)<\rho}}\,
|\nabla{f}|_{\rho}(x,y)
\end{gather}
and are called, respectively, the \emph{$\rho$-slope} of $f$ at $(x,y)$ ($f(x,y)<\infty$) and the \emph{strict outer slope} of $f$ at $(\bx,\by)$.
Definition \eqref{ls-f} is a realization of the local (strong) \emph{slope} \cite{DMT} for the case of a function on a product space with the product metric defined by \eqref{drho}.

The next \emph{modified strict outer slope} of $f$ at $(\bx,\by)$ can be useful when estimating the error bound modulus:
\begin{gather}\label{mss-f}
\overline{|\nabla{f}|}{}^{>+}(\bar{x},\by):=
\lim_{\rho\downarrow0}
\inf_{\substack{d(x,\bx)<\rho,\,0<f(x,y)<\rho}}\,
\max\left\{|\nabla{f}|_{\rho}(x,y), \frac{f(x,y)}{d(x,\bx)}\right\}.
\end{gather}
Observe that, unlike the completely local constant \eqref{ss-f}, its modified counterpart \eqref{mss-f} contains under $\max$ a nonlocal (when $(x,y)$ is fixed) component $f(x,y)/d(x,\bx)$.

In normed linear spaces, one can use for estimating slopes and hence error bounds other tools based on either directional derivatives or subdifferentials.
Below we describe some tools from the second group.

Suppose $X$ and $Y$ are normed linear spaces.
In the product space $X\times Y$, along with the usual maximum-type norm
\begin{gather*}
\|(u,v)\|=\max\{\|u\|,\|v\|\},\quad (u,v)\in X\times Y,
\end{gather*}
we consider the $\rho$-norm $\|\cdot\|_\rho$ being the realization of the $\rho$-metric \eqref{drho}:
\begin{gather*}\label{nrho}
\|(u,v)\|_\rho=\max\{\|u\|,\rho\|v\|\},\quad (u,v)\in X\times Y.
\end{gather*}
The corresponding dual norm (we keep the same notation $\|\cdot\|_\rho$ for it) is of the form:
\begin{gather*}\label{nrho*}
\|(u^*,v^*)\|_\rho=\|u^*\|+\rho\iv\|v^*\|,\quad (u^*,v^*)\in X^*\times Y^*.
\end{gather*}

One can define subdifferential counterparts of the local slopes \eqref{ls-f}--\eqref{mss-f}:
the \emph{subdifferential $\rho$-slope} \begin{equation}\label{sds-f}
|\sd{f}|_{\rho}(x,y):=
\inf_{\substack{(x^*,y^*)\in\sd f(x,y),\,
\|y^*\|<\rho}} \|x^*\|
\end{equation}
of $f$ at $(x,y)$ ($f(x,y)<\infty$) and the \emph{strict outer} and, respectively, \emph{modified strict outer subdifferential slopes}
\begin{gather}\label{ssds-f}
\overline{|\sd{f}|}{}^{>}(\bar{x},\by):=
\lim_{\rho\downarrow0}
\inf_{\substack{\|x-\bx\|<\rho,\,0<f(x,y)<\rho}}\,
|\sd{f}|_{\rho}(x,y),
\\\label{mssds-f}
\overline{|\sd{f}|}{}^{>+}(\bar{x},\by):=
\lim_{\rho\downarrow0}
\inf_{\substack{\|x-\bx\|<\rho,\,0<f(x,y)<\rho}}
\max\left\{|\sd{f}|_{\rho}(x,y), \frac{f(x,y)}{\|x-\bx\|}\right\}
\end{gather}
of $f$ at
$(\bx,\by)$.

\begin{remark}\label{RM3}
Definitions \eqref{uss-f}, \eqref{ss-f}, \eqref{mss-f}, \eqref{ssds-f}, and \eqref{mssds-f} corresponding to the lower $0$-le\-vel set $S(f)$ of $f$ can be easily extended to the case of the general lower level set $\{x\in X\mid f(x,\by)\le f(\bx,\by)\}$ with an arbitrary finite $f(\bx,\by)$.
It is sufficient to replace $f$ in (P1), (P2), \eqref{nls-f}, \eqref{uss-f}, \eqref{ss-f}, \eqref{mss-f}, \eqref{ssds-f}, and \eqref{mssds-f} with function $(x,y)\mapsto f(x,y)-f(\bx,\by)$.
\end{remark}

\begin{remark}
$\rho$-slopes \eqref{nls-f}, \eqref{ls-f}, and \eqref{sds-f} are nondecreasing functions of $\rho$.
This makes the infimums in \eqref{uss-f}, \eqref{ss-f}, \eqref{mss-f}, \eqref{ssds-f}, and \eqref{mssds-f} nondecreasing functions of $\rho$ too.
This observation allows one to simplify estimates involving strict slopes, e.g., in Corollary~\ref{3C1.1} below.
\end{remark}

The next proposition summarizes the relationships between the slopes.
It extends \cite[Proposition 3.2 and Theorem 3.3]{Kru15}.

\begin{proposition}[Relationships between slopes]\label{nc}
\begin{enumerate}
\item
$\ds|\nabla{f}|_{\rho}^{\diamond}(x,y)\ge
\max\left\{|\nabla{f}|_{\rho}(x,y), \frac{f(x,y)}{d_\rho((x,y),(\bx,\by))}\right\}$\\
for all $\rho>0$ and all $(x,y)\in X\times Y$ with $0<f(x,y)<\infty$;
\item
$\ds\overline{|\nabla{f}|}{}^{\diamond}(\bar{x},\by)\ge
\overline{|\nabla{f}|}{}^{>+}(\bx,\by)\ge
\overline{|\nabla{f}|}{}^{>}(\bx,\by)$.
\cnta
\end{enumerate}
Suppose $X$ and $Y$ are normed linear spaces.
\begin{enumerate}
\cntb
\item
$\ds\overline{|\sd{f}|}{}^{>+}(\bx,\by)\ge
\overline{|\sd{f}|}{}^{>}(\bx,\by)$;
\item
$|\nabla{f}|_{\rho}(x,y)\le
|\sd{f}|_{\rho^2}(x,y)+\rho$\\ for all $\rho>0$ and all $(x,y)\in X\times Y$ with $f(x,y)<\infty$;
\item
$\overline{|\nabla{f}|}{}^{>}(\bar{x},\by)\le
\overline{|\sd{f}|}{}^{>}(\bx,\by)$ and
$\overline{|\nabla{f}|}{}^{>+}(\bar{x},\by)\le
\overline{|\sd{f}|}{}^{>+}(\bx,\by)$;
\item\label{nc.5}
$\overline{|\nabla{f}|}{}^{>}(\bar{x},\by)=
\overline{|\sd{f}|}{}^{>}(\bx,\by)$ and $\overline{|\nabla{f}|}{}^{>+}(\bar{x},\by)=
\overline{|\sd{f}|}{}^{>+}(\bx,\by)$\\
provided that one of the following conditions is satisfied:
\begin{enumerate}
\item
$X$ and $Y$ are Asplund and $f_+$ is \lsc\ near $(\bar{x},\by)$;
\item
$f$ is convex near $(\bar{x},\by)$;
\item
$f$ is Fr\'echet differentiable near $(\bar{x},\by)$ except $(\bx,\by)$;
\item
$f=f_1+f_2$, where $f_1$ is convex near $(\bar{x},\by)$ and $f_2$ is Fr\'echet differentiable near $(\bar{x},\by)$ except $(\bx,\by)$;
\end{enumerate}
\item
if $f$ is convex near $(\bx,\by)$, then \begin{enumerate}
\item
$|\nabla{f}|_{\rho}^{\diamond}(x,y)= |\nabla{f}|_{\rho}(x,y)$\\
for all $\rho>0$ and all $(x,y)\in X\times Y$ near $(\bx,\by)$ with $0<f(x,y)<\infty$;
\item
$\ds\overline{|\nabla{f}|}{}^{\diamond}(\bar{x},\by)=
\overline{|\nabla{f}|}{}^{>+}(\bx,\by)=
\overline{|\nabla{f}|}{}^{>}(\bx,\by)= \overline{|\sd{f}|}{}^{>+}(\bx,\by)= \overline{|\sd{f}|}{}^{>}(\bx,\by)$.
\end{enumerate}
\end{enumerate}
\end{proposition}

\begin{proof}
(i) Inequality $|\nabla{f}|_{\rho}^{\diamond}(x,y)\ge
|\nabla{f}|_{\rho}(x,y)$ comes from \cite[Proposition 3.2(i)]{Kru15}, while
the other inequality is a direct consequence of definition \eqref{nls-f}.

(ii) The second inequality follows from comparing definitions \eqref{ss-f} and \eqref{mss-f}, while
the first inequality requires a modification of the proof of \cite[Proposition 3.2(iii)]{Kru15}.
Let $\overline{|\nabla{f}|}{}^{\diamond}(\bar{x},\by) <\ga<\infty$ and $\rho>0$.
By (P2), one can find a $\rho'\in(0,\rho)$ such that \begin{align}\label{ndi}
\frac{f(x,y)}{d(y,\by)}>\rho'\ga
\end{align}
as long as $0<f(x,y)<\rho'$.
By \eqref{uss-f}, there exists a point $(x,y)\in X\times Y$ with $d(x,\bx)<\rho'$ and $0<f(x,y)<\rho'$ such that $|\nabla{f}|{}^{\diamond}_{\rho'}(x,y)<\ga$, i.e., by \eqref{nls-f},
$$
\frac{f(x,y)-f_+(u,v)}{d_{\rho'}((x,y),(u,v))}<\ga
$$
for all $(u,v)\ne(x,y)$.
Observe that $(x,y)\ne(\bx,\by)$ since $f(x,y)>0$.
Hence, $|\nabla{f}|_{\rho}(x,y)\le\ga$ and
$$
\frac{f(x,y)}{d_{\rho'}((x,y),(\bx,\by))} =\min\left\{\frac{f(x,y)}{d(x,\bx)}, (\rho')\iv\frac{f(x,y)}{d(y,\by)}\right\}<\ga.
$$
Together with \eqref{ndi}, this implies $f(x,y)/d(x,\bx)<\ga$.
Thus,
$$
\max\left\{|\nabla{f}|_{\rho}(x,y), \frac{f(x,y)}{d(x,\bx)}\right\}\le\ga
$$
and consequently,
\begin{align*}
\inf_{\substack{d(x,\bx)<\rho,\,0<f(x,y)<\rho}}\,
\max\left\{|\nabla{f}|_{\rho}(x,y), \frac{f(x,y)}{d(x,\bx)}\right\}\le\ga.
\end{align*}
Taking limits as $\rho\downarrow0$ and $\ga\downarrow \overline{|\nabla{f}|}{}^{\diamond}(\bar{x},\by)$, we arrive at the claimed inequality.

(iii) is a consequence of definitions \eqref{ssds-f} and \eqref{mssds-f}.

(iv) comes from \cite[Theorem 3.3(i)]{Kru15}.

(v) is a consequence of (iv) and definitions \eqref{ssds-f} and \eqref{mssds-f}.

(\ref{nc.5}) Under condition (a), the first equality comes from \cite[Theorem 3.3(iii)]{Kru15}, while the second one requires a modification of the proof of that theorem.

Let $X$ and $Y$ be Asplund and $f_+$ be \lsc\ near $(\bar{x},\by)$.
Thanks to (v), we only need to prove that
$\overline{|\nabla{f}|}{}^{>+}(\bar{x},\by)\ge
\overline{|\sd{f}|}{}^{>+}(\bx,\by)$.
If $\overline{|\nabla{f}|}{}^{>+}(\bar{x},\by)=\infty$, the assertion is trivial.
Let $\overline{|\nabla{f}|}{}^{>+}(\bar{x},\by)< \gamma<\infty$.
Choose a $\ga'\in(\overline{|\nabla{f}|}{}^{>+}(\bar{x},\by),\gamma)$ and an arbitrary $\rho>0$.
Set $\rho':=\min\{1,\ga\iv\}\rho$.
By definitions \eqref{mss-f} and \eqref{ls-f},
one can find a point $(x,y)\in X\times Y$ such that $d(x,\bx)<\rho'$, $0<f(x,y)<\rho'$, $f(x,y)/d(x,\bx)<\ga'$, $f$ is \lsc\ near $(x,y)$, and
$$
f(x,y)-f(u,v)\le\ga'\|(u,v)-(x,y)\|_{\rho'}\quad \mbox{for all } (u,v) \mbox{ near } (x,y).
$$
In other words, $(x,y)$ is a point of local minimum of the function $$(u,v)\mapsto f(u,v)+\ga'\|(u,v)-(x,y)\|_{\rho'}.$$
Take an
$$\eps\in\left(0,\min\set{\rho-d(x,\bx),\rho-f(x,y), \ga-\ga',\frac{\ga d(x,\bx)-f(x,y)}{\ga+1}}\right)$$
sufficiently small such that $f$ is \lsc\ on $B_\eps((x,y))$ and $B_\eps(x)\cap S(f)=\emptyset$.
Applying the \emph{fuzzy sum rule} (Lemma~\ref{l02}(i)), we find points $(z,w)\in X\times Y$ and $(x^*,y^*)\in\sd f(z,w)$ such that 
$$d((z,w),(x,y))<\eps,\; f(z,w)<f(x,y)+\eps,\mbox{ and } \|(x^*,y^*)\|_{\rho'}<\ga'+\eps.$$
It follows that
$$d(z,\bx)<\rho,\;0<f(z,w)<\rho,\;f(z,w)/d(z,\bx)<\ga,\; \|x^*\|<\ga,\mbox{ and }\|y^*\|<{\rho'}\ga\le\rho.$$
Hence, $|\sd{f}|_{\rho}(z,w)<\ga$ and consequently $\overline{|\sd{f}|}{}^{>+}(\bar{x},\by)\le\ga$.
The claimed inequality follows after letting $\ga\to\overline{|\nabla{f}|}{}^{>+}(\bar{x},\by)$.

The validity of the equalities under each of the conditions (b)--(d) follows from the proof of \cite[Theorem 3.3(iii)]{Kru15} if one replaces the fuzzy sum rule for Fr\'echet subdifferentials in Asplund spaces there with the exact convex sum rule (Lemma~\ref{l02}(iii)) and differentiable sum rule (Lemma~\ref{l02}(ii)) valid in arbitrary normed spaces.

(vii) Let $\rho>0$, $(x,y)$ belong to a convex neighbourhood of $(\bx,\by)$ on which $f$ is convex and $0<f(x,y)<\infty$.
Take any $(u,v)\in X\times Y$ with $f(u,v)<f(x,y)$.
Then,
\begin{gather*}
\frac{f(x,y)-f_+(u,v)}{d_\rho((x,y),(u,v))}\le \lim_{t\downarrow0} \frac{f(x,y)-f((1-t)(x,y)+t(u,v))} {d_\rho((1-t)(x,y)+t(u,v),(x,y))}\le |\nabla{f}|_{\rho}(x,y),
\end{gather*}
and consequently, $|\nabla{f}|_{\rho}^{\diamond}(x,y)\le |\nabla{f}|_{\rho}(x,y)$.
Together with (i), this proves (a).
Equality
$\ds\overline{|\nabla{f}|}{}^{\diamond}(\bar{x},\by)=
\overline{|\nabla{f}|}{}^{>}(\bx,\by)$ follows from (a) and definitions \eqref{uss-f} and \eqref{ss-f}.
The other equalities in (b) follow from (ii) and (\ref{nc.5})(b).
\qed\end{proof}

\begin{remark}\label{R4}
One of the main tools in the proof of inequalities
\begin{align*}
\overline{|\nabla{f}|}{}^>(\bar{x},\by)
\ge\overline{|\partial{f}|}{}^>(\bar{x},\by),\qquad
\overline{|\nabla{f}|}{}^{>+}(\bar{x},\by)
\ge\overline{|\partial{f}|}{}^{>+}(\bar{x},\by)
\end{align*}
in item (a) of part (\ref{nc.5}) of the above proposition, which is crucial for the subdifferential sufficient error bound criteria in Corollaries~\ref{3C1.1} and \ref{C1-0} below, is the fuzzy sum rule (Lemma~\ref{l02}) for Fr\'echet subdifferentials in Asplund spaces.
There are two ways of extending this inequality to general Banach spaces.

1) Restricting the class of functions to those possessing a kind of sum rule for Fr\'echet subdifferentials as in items (b)--(d) of part (\ref{nc.5}) of the above proposition.

2) Replacing Fr\'echet subdifferentials with some other (possibly abstract) subdifferentials on the given space satisfying a certain set of natural properties including a kind of (fuzzy or exact) sum rule (\emph{trustworthy} subdifferentials \cite{Iof83,Iof00_}); cf. \cite[Proposition~1.13]{Aze03}, \cite[Proposition~2.3]{Aze06}, \cite[Proposition~4.1]{AzeCor04}, \cite[Proposition~6]{FabHenKruOut10}.
For instance, one can use for that purpose \emph{Ioffe approximate} or \emph{Clarke-Rockafellar} subdifferentials.
Note that the opposite inequalities in part (v) are specific for Fr\'echet subdifferentials and fail in general for other types of subdifferentials.
\end{remark}

The uniform strict outer slope \eqref{uss-f} provides the necessary and sufficient characterization of error bounds \cite[Theorem 4.1]{Kru15}.

\begin{theorem}\label{3T1}
\begin{enumerate}
\item
$\Er f(\bx,\by)\le
\overline{|\nabla{f}|}{}^{\diamond}(\bar{x},\by)$;
\item
if $X$ and $Y$ are complete and $f_+$ is lower semicontinuous
near $(\bar{x},\by)$,
then
$\Er f(\bx,\by)=
\overline{|\nabla{f}|}{}^{\diamond}(\bar{x},\by)$.
\end{enumerate}
\end{theorem}

It follows from Theorem~\ref{3T1}  that inequality
$\overline{|\nabla{f}|}{}^{\diamond}(\bar{x},\by)>0$ is crucial for determining the error bound property of $f$ at $(\bx,\by)$.

\begin{remark}\label{rm2}
The nonlocal $\rho$-slope \eqref{nls-f} depends on the choice of $\rho$-metric on the product space.
If instead of the maximum-type metric $d_\rho$, defined by \eqref{drho}, one employs in \eqref{nls-f} the sum-type metric $d_\rho^1$, defined by \eqref{drho1},
it will produce a different number.
We say that a $\rho$-metric $d'_\rho$ on $X\times Y$ is admissible if $d_\rho\le d'_\rho\le d^1_\rho$.
Thanks to \cite[Proposition~4.2]{Kru15}, Theorem~\ref{3T1} is invariant on the choice of an admissible $\rho$-metric.
\end{remark}

Thanks to Theorem~\ref{3T1} and Proposition~\ref{nc}, one can formulate several quantitative and qualitative criteria of the error bound property in terms of various slopes.

\begin{corollary}[Quantitative criteria]\label{3C1.1}
Let $\ga>0$.
Consider the following conditions:
\renewcommand {\theenumi} {\alph{enumi}}
\begin{enumerate}
\item
$f$ has an error bound with respect to $x$ at $(\bx,\by)$ with some $\tau>0$;
\item
$\overline{|\nabla{f}|}{}^{\diamond}(\bar{x},\by)>\ga$,\\ i.e., for some $\rho>0$ and any $(x,y)\in X\times Y$ with $d(x,\bx)<\rho$ and $0<f(x,y)<\rho$, it holds $|\nabla{f}|_\rho^{\diamond}(x,y)>\ga$, and consequently there is a $(u,v)\in X\times Y$ such that
\begin{gather*}\label{3ree}
f(x,y)-f_+(u,v)>\ga d_\rho((u,v),(x,y));
\end{gather*}
\item
$\ds\liminf_{x\to\bar{x},\;f(x,y)\downarrow0} \frac{f(x,y)}{d(x,\bx)}>\ga$;
\item
$\overline{|\nabla{f}|}{}^{>}(\bar{x},\by)>\ga$,\\ i.e., for some $\rho>0$ and any $(x,y)\in X\times Y$ with $d(x,\bx)<\rho$ and $0<f(x,y)<\rho$, it holds $|\nabla{f}|_\rho(x,y)>\ga$ and consequently, for any $\eps>0$, there is a $(u,v)\in B_\eps(x,y)$ such that
\begin{gather}\label{3ree2}
f(x,y)-f(u,v)>\ga d_\rho((u,v),(x,y));
\end{gather}
\item
$\overline{|\nabla{f}|}{}^{>+}(\bar{x},\by)>\ga$,\\ i.e., for some $\rho>0$ and any $(x,y)\in X\times Y$ with $d(x,\bx)<\rho$, $0<f(x,y)<\rho$, and $f(x,y)/d(x,\bx)\le\ga$, it holds $|\nabla{f}|_\rho(x,y)>\ga$ and consequently, for any $\eps>0$, there is a $(u,v)\in B_\eps(x,y)$ such that \eqref{3ree2} holds true;
\item
$X$ and $Y$ are normed spaces and $\overline{|\sd{f}|}{}^{>}(\bar{x},\by)>\ga$,\\ i.e.,
for some $\rho>0$ and any $(x,y)\in X\times Y$ with $\|x-\bx\|<\rho$ and $0<f(x,y)<\rho$, it holds $|\sd{f}|_\rho(x,y)>\ga$ and consequently $\|x^*\|>\ga$ for all $(x^*,y^*)\in\sd f(x,y)$ with $\|y^*\|<\rho$;
\item
$X$ and $Y$ are normed spaces and $\overline{|\sd{f}|}{}^{>+}(\bar{x},\by)>\ga$,\\  i.e.,
for some $\rho>0$ and any $(x,y)\in X\times Y$ with $\|x-\bx\|<\rho$, $0<f(x,y)<\rho$, and $f(x,y)/\|x-\bx\|\le\ga$, it holds $|\sd{f}|_\rho(x,y)>\ga$ and consequently $\|x^*\|>\ga$ for all $(x^*,y^*)\in\sd f(x,y)$ with $\|y^*\|<\rho$.
\end{enumerate}
\renewcommand {\theenumi} {\roman{enumi}}
The following implications hold true:
\begin{enumerate}
\item
{\rm (c) \folgt (e)},
{\rm (d) \folgt (e)},
{\rm (e) \folgt (b)},
{\rm (f) \folgt (g)};
\item
if $\ga<\tau$, then {\rm (a) \folgt (b)};
\item
if $\tau\le\ga$, $X$ and $Y$ are complete, and $f_+$ is lower semicontinuous near $(\bar{x},\by)$, then {\rm (b) \folgt (a)}.
\cnta
\end{enumerate}
Suppose $X$ and $Y$ are normed spaces.
\begin{enumerate}
\cntb
\item
{\rm (d) \folgt (f)} and {\rm (e) \folgt (g)};
\item
{\rm (d)~\iff (f) \folgt (e)~\iff (g)} provided that one of the conditions {\rm (a)--(d)} in part {\rm (\ref{nc.5})} of Proposition~\ref{nc} is satisfied;
\item
if $f$ is convex near $(\bx,\by)$, then {\rm (b)~\iff (d) \iff (e)~\iff (f)~\iff (g)}.
\end{enumerate}
\end{corollary}

The conclusions of Corollary~\ref{3C1.1} are illustrated in
Fig.~\ref{fig.1}.

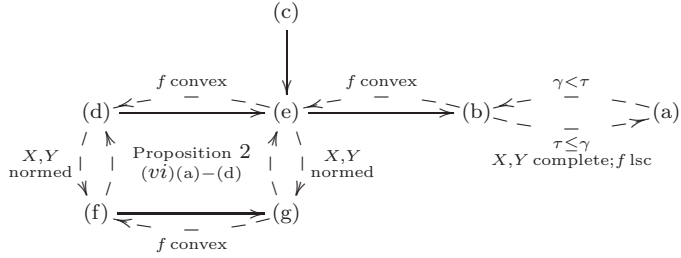
\begin{figure}[!htb]
$$\xymatrix{
&&{\rm (c)}\ar[d]
\\
{\rm (d)}\ar[rr]
\ar@/_/@{-->}[d]_{\substack{X,Y\\{\rm normed}}}
\ar @{} [drr] |{\substack{{\rm Proposition~\ref{nc}}\\{\rm (\ref{nc.5})(a)-(d)}}}
&&{\rm (e)}\ar[rr]
\ar@/^/@{-->}[d]^{\substack{X,Y\\{\rm normed}}}
\ar@/_/@{-->}[ll]_{f\,{\rm convex}}
&&{\rm (b)} \ar@/_/@{-->}[rr]_{\substack{\tau\le\ga\\X,Y\, {\rm complete}; f\,{\rm lsc}}}
\ar@/_/@{-->}[ll]_{f\,{\rm convex}}
&&{\rm (a)} \ar@/_/@{-->}[ll]_{\ga<\tau}
\\
{\rm (f)}\ar[rr]
\ar@/_/@{-->}[u]
&&{\rm (g)}
\ar@/^/@{-->}[u]
\ar@/^/@{-->}[ll]^{f\,{\rm convex}}
}$$
\caption{Corollary~\ref{3C1.1}} \label{fig.1}
\end{figure}

\begin{corollary}[Qualitative criteria]\label{C1-0}
Suppose $X$ and $Y$ are complete metric spaces and $f_+$ is lower semicontinuous near $(\bar{x},\by)$.
Then,
$f$ has an error bound with respect to $x$ at $(\bx,\by)$ provided that one of the following conditions holds true:
\renewcommand {\theenumi} {\alph{enumi}}
\begin{enumerate}
\item
$\overline{|\nabla{f}|}{}^{\diamond}(\bar{x},\by)>0$; \item
$\ds\liminf_{x\to\bar{x},\;f(x,y)\downarrow0} \frac{f(x,y)}{d(x,\bx)}>0$;
\item
$\overline{|\nabla{f}|}{}^{>}(\bar{x},\by)>0$;
\item
$\overline{|\nabla{f}|}{}^{>+}(\bar{x},\by)>0$, or equivalently,
$$\ds\lim_{\rho\downarrow0}
\inf_{\substack{d(x,\bx)<\rho,\,0<\frac{f(x,y)}{d(x,\bx)}<\rho}}\,
|\nabla{f}|_\rho(x,y)>0.$$
\cnta
\end{enumerate}
If $X$ and $Y$ are Banach spaces and
one of the conditions {\rm (a)--(d)} in part {\rm (\ref{nc.5})} of Proposition~\ref{nc} is satisfied, then the following conditions are also sufficient:
\begin{enumerate}
\cntb
\item
$\overline{|\sd{f}|}{}^{>}(\bar{x},\by)>0$;
\item
$\overline{|\sd{f}|}{}^{>+}(\bar{x},\by)>0$,
or equivalently,
$$\ds\lim_{\rho\downarrow0}
\inf_{\substack{\|x-\bx\|<\rho,\,0<\frac{f(x,y)}{\|x-\bx\|}<\rho}}\,
|\sd{f}|_\rho(x,y)>0.$$
\end{enumerate}
\renewcommand {\theenumi} {\roman{enumi}}
Moreover,
\begin{enumerate}
\item
condition {\rm (a)} is also necessary for the local error bound property of $f$ at $(\bx,\by)$;
\item
{\rm (b) \folgt (d)},
{\rm (c) \folgt (d)},
{\rm (d) \folgt (a)},
{\rm (e) \folgt (f)}.
\cnta
\end{enumerate}
Suppose $X$ and $Y$ are Banach spaces and
one of the conditions {\rm (a)--(d)} in part {\rm (\ref{nc.5})} of Proposition~\ref{nc} is satisfied.
\begin{enumerate}
\cntb
\item
{\rm (c) \iff (e)} and {\rm (d) \iff (f)};
\item
if $f$ is convex near $(\bx,\by)$, then {\rm (a) \iff (c) \iff (d) \iff \rm (e) \iff (f)}.
\end{enumerate}
\end{corollary}

The conclusions of Corollary~\ref{C1-0} are illustrated in
Fig.~\ref{fig.2}.

\begin{figure}[!htb]
$$\xymatrix@C=.8cm{
&*+[F]{\ds\liminf_{x\to\bar{x},f(x,y)\downarrow0} \frac{f(x,y)}{d(x,\bx)}>0}\ar[d]
\\
*+[F]{\overline{|\nabla{f}|}{}^{>}(\bar{x},\by)>0}
\ar[r]
\ar@{-->}[d]
\ar@{} [dr] |{\substack{X,Y\,{\rm Banach}\\{\rm Proposition~\ref{nc}}\\{\rm (\ref{nc.5})(a)-(d)}}}
&*+[F]{\overline{|\nabla{f}|}{}^{>+}(\bar{x},\by)>0}
\ar[r]
\ar@{-->}[d]
\ar@/_/@{-->}[l]_{f\,{\rm convex}}
&*+[F]{\overline{|\nabla{f}|}{}^{\diamond}(\bar{x},\by)>0}
\ar[r]
\ar@/_/@{-->}[l]_{f\,{\rm convex}}
&*+[F]{\Er f(\bx,\by)>0} \ar[l]
\\
*+[F]{\overline{|\sd{f}|}{}^{>}(\bar{x},\by)>0}
\ar[r]
\ar@{-->}[u]
&*+[F]{\overline{|\sd{f}|}{}^{>+}(\bar{x},\by)>0}
\ar@{-->}[u]
\ar@/^/@{-->}[l]^{f\,{\rm convex}}
}$$
\caption{Corollary~\ref{C1-0}}\label{fig.2}
\end{figure}
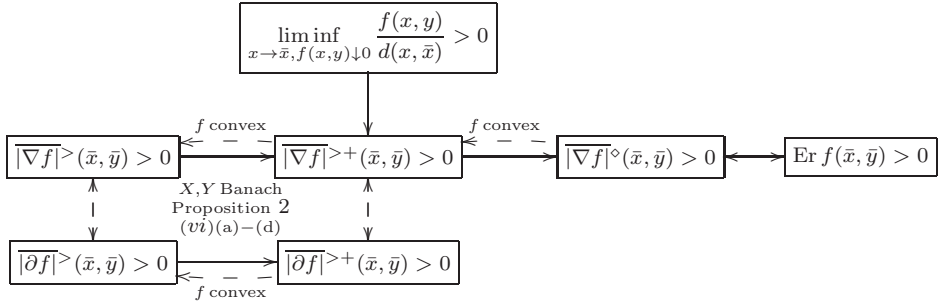

\section{H\"older metric subregularity}\label{S5}

From now on, $F:X\rightrightarrows Y$ is a set-valued mapping between metric spaces and $(\bx,\by)\in\gph F$.
We are targeting the H\"older metric subregularity property \eqref{02}, the main tool being the error bound criteria discussed in the previous section.
First, recall the definition.

\subsection{Definition}
Let a real number $q\in(0,1]$ be given.

A set-valued mapping $F:X\rightrightarrows Y$ between metric spaces is called H\"older metrically subregular of order $q$ at $(\bx,\by)\in\gph F$ with constant $\tau>0$ if there exists a neighbourhood $U$ of $\bx$ such that
\begin{equation}\label{HMR}
\tau d(x,F^{-1}(\by))\le (d(\by,F(x)))^q\quad \mbox{for all } x\in U,
\end{equation}
or equivalently,
\begin{equation*}
\tau d(x,F^{-1}(\by))\le (d(y,\by))^q\quad \mbox{for all } x\in U,\; y\in F(x).
\end{equation*}

The H\"older metric subregularity property can be characterized using the following (possibly infinite) constant:
\begin{gather}\label{7CMR-g}
\sr_q[F](\bx,\by):=
\liminf_{\substack{x\to\bx\\x\notin F\iv(\by)}} \frac{(d(\by,F(x)))^q}{d(x,F^{-1}(\by))},
\end{gather}
which coincides with the supremum of all positive $\tau$ such that \eqref{HMR} holds for some $U$ and, when positive, provides a quantitative characterization of this property.

Property \eqref{HMR} is exactly the error bound property \eqref{eb} for the function $f$ defined by
\begin{gather}\label{f}
f(x,y):=
\begin{cases}
(d(y,\by))^q & \text{if } (x,y)\in\gph F,\\
+\infty & \text{otherwise},
\end{cases}
\end{gather}
while constant \eqref{7CMR-g} coincides \eqref{rr2}.
Indeed, $f(\bx,\by)=0$,
conditions (P1) and (P2) are trivially satisfied, and $f(x,y)=0$ if and only if $y=\by$ and $x\in F^{-1}(\by)$.
Hence, $S(f)=F^{-1}(\by)$.
Observe also that condition $f(x,y)\downarrow0$ is equivalent to $y\to\by$ with $(x,y)\in\gph F$ and $y\ne\by$.

\subsection{Primal space slopes}
The nonlocal slopes \eqref{nls-f} and \eqref{uss-f} of $f$ in the current setting take the following form:
\begin{gather}\label{7nls}
|\nabla{F}|_{q,\rho}^{\diamond}(x,y):=
\sup_{\substack{(u,v)\ne(x,y)\\(u,v)\in\gph F}}
\frac{[(d(y,\by))^q-(d(v,\by))^q]_+}{d_\rho((u,v),(x,y))},
\\\label{7uss}
\overline{|\nabla{F}|}{}^{\diamond}_q(\bar{x},\by):=
\lim_{\rho\downarrow0}
\inf_{\substack{d(x,\bx)<\rho,\,d(y,\by)<\rho\\
(x,y)\in\gph F,\,x\notin F\iv(\by)}}\,
|\nabla{F}|{}^{\diamond}_{q,\rho}(x,y).
\end{gather}
We will call the above constants, respectively,
the \emph{nonlocal $(q,\rho)$-slope} of $F$ at $(x,y)\in\gph F$ and
the \emph{uniform strict $q$-slope} of $F$ at $(\bx,\by)$.

As the main primal space local tool, in this section we are going to use the \emph{$\rho$-slope} of $F$ at $(x,y)$:
\begin{gather}\label{srho}
|\nabla{F}|_{\rho}(x,y):=
\limsup_{\substack{(u,v)\to(x,y),\,
(u,v)\ne(x,y)\\
(u,v)\in\gph F}}
\frac{[d(y,\by)-d(v,\by)]_+} {d_\rho((u,v),(x,y))}
\end{gather}
defined for all $\rho>0$ and $(x,y)\in\gph F$.
Observe that constant \eqref{srho} does not depend on $q$.

Using some simple calculus, one can formulate the representation for the local $\rho$-slo\-pe \eqref{ls-f} in the case when $f$ is given by \eqref{f}.

\begin{proposition}\label{P3}
Suppose $(x,y)\in\gph F$, $y\ne\by$, $\rho>0$, and $f$ is given by \eqref{f}.
Then,
$$|\nabla{f}|_{\rho}(x,y)
=q(d(y,\by))^{q-1}\;|\nabla{F}|_{\rho}(x,y).$$
\end{proposition}

\begin{proof}
By \eqref{ls-f}, \eqref{f}, and \eqref{srho},
\begin{align*}
|\nabla{f}|_{\rho}(x,y)
&=\limsup_{\substack{(u,v)\to(x,y),\,
(u,v)\ne(x,y)\\(u,v)\in\gph F}}
\frac{[(d(y,\by))^q-(d(v,\by))^q]_+} {d_\rho((u,v),(x,y))}
\\
&=\limsup_{\substack{(u,v)\to(x,y),\,
(u,v)\ne(x,y)\\(u,v)\in\gph F}}
\frac{[q(d(y,\by))^{q-1}(d(y,\by)-d(v,\by))+o(d(v,y))]_+} {d_\rho((u,v),(x,y))}
\\
&=q(d(y,\by))^{q-1} \limsup_{\substack{(u,v)\to(x,y),\,
(u,v)\ne(x,y)\\(u,v)\in\gph F}}
\frac{[d(y,\by)-d(v,\by)]_+} {d_\rho((u,v),(x,y))}
\\
&=q(d(y,\by))^{q-1}\;|\nabla{F}|_{\rho}(x,y).
\end{align*}
In the above formula, $o(\cdot)$ is a function from $\R_+$ to $\R_+$ with the property $o(t)/t\to0$ as $t\downarrow0$.
\qed\end{proof}
The \emph{strict $q$-slope} and the \emph{modified strict $q$-slope} of $F$ at $(\bx,\by)$ are defined as follows:
\begin{gather}\label{q-ss}
\overline{|\nabla{F}|}{}_q(\bar{x},\by):=
q\lim_{\rho\downarrow0}
\inf_{\substack{d(x,\bx)<\rho,\,d(y,\by)<\rho\\
(x,y)\in\gph F,\,x\notin F\iv(\by)}}\,
(d(y,\by))^{q-1}\;|\nabla{F}|_{\rho}(x,y),
\\\label{aq-ss}
\overline{|\nabla{F}|}{}_q^{+}(\bar{x},\by):=
\lim_{\rho\downarrow0}
\inf_{\substack{d(x,\bx)<\rho,\,d(y,\by)<\rho\\
(x,y)\in\gph F,\,x\notin F\iv(\by)}}
\max\left\{q(d(y,\by))^{q-1}|\nabla{F}|_{\rho}(x,y), \frac{(d(y,\by))^q}{d(x,\bx)}\right\}.
\end{gather}

In view of Proposition~\ref{P3}, these constants coincide, respectively, with the corresponding strict outer slopes \eqref{ss-f} and \eqref{mss-f} of $f$.

\subsection{Subdifferential slopes}
In this subsection, $X$ and $Y$ are normed spaces.
We define, respectively, the \emph{subdifferential $\rho$-slope} and the \emph{approximate subdifferential $\rho$-slope} ($\rho>0$) of $F$ at $(x,y)\in\gph F$ with $y\ne\by$ as
\begin{gather}\label{6srs}
|\sd{F}|_{\rho}(x,y)
:=\inf_{\substack{x^*\in D^*F(x,y)(J(y-\by)+\rho\B^*)}}
\|x^*\|,
\\\label{6asrs}
|\sd{F}|^a_{\rho}(x,y)
:=\liminf_{\substack{v\to y-\by}}\
\inf_{\substack{x^*\in D^*F(x,y)(J(v)+\rho\B^*)}}
\|x^*\|,
\end{gather}
where $J$ is the duality mapping defined by \eqref{J}.
These constants do not depend on $q$.

In the rest of the section, when $Y$ is a normed space, we use the notation
$$\xi_q(y):=\|y-\by\|^{1-q}/q.$$

The next proposition gives representations for the subdifferential $\rho$-slope \eqref{sds-f} in the case when $f$ is given by \eqref{f}.

\begin{proposition}\label{6P1}
Suppose $(x,y)\in\gph F$, $y\ne\by$, $\rho>0$, and $f$ is given by \eqref{f}.
\begin{enumerate}
\item
If $X$ and $Y$ are Asplund and $\gph F$ is locally closed near $(x,y)$, then
\begin{gather*}
\hspace{-.3cm}
\ds|\sd{f}|_{\rho}(x,y)
\ge q\|y-\by\|^{q-1}\liminf_{\substack{ (x',y')\to(x,y),\,y''\to y\\(x',y')\in\gph F}}\; \inf_{\substack{x^*\in D^*F(x',y')(J(y''-\by)+\xi_q(y'')\rho\B^*)}} \|x^*\|.
\end{gather*}
\item
If either
$F$ is convex near $(x,y)$ and $q=1$ or $Y$ is Fr\'echet smooth, then
$$\ds|\sd{f}|_{\rho}(x,y)
=q\|y-\by\|^{q-1}\;|\sd{F}|_{\xi_q(y)\rho}(x,y).$$
\end{enumerate}
\end{proposition}

\begin{proof}
(i) Suppose that $X$ and $Y$ are normed spaces and observe that
\begin{gather*}\label{gF2}
f(u,v)=g(v)+ i_{\gph F}(u,v),\quad (u,v)\in X\times Y,
\end{gather*}
where $g(v)=\|v-\by\|^q$ and $i_{\gph F}$ is the \emph{indicator function} of $\gph F$: $i_{\gph F}(u,v)=0$ if $(u,v)\in\gph F$ and $i_{\gph F}(u,v)=\infty$ otherwise.
In its turn, function $g$ is a composition of two functions: $v\mapsto \|v-\by\|$ on $Y$ and $t\mapsto t^q$ on $\R_+$.
The latter function is Fr\'echet differentiable on $(0,\infty)$.
It follows from the composition rule for Fr\'echet subdifferentials (cf., e.g., \cite[Corollary~1.14.1]{Kru03.1}) that, for any $v\ne\by$,
\begin{gather*}
\sd g(v)=q\|v-\by\|^{q-1}J(v-\by).
\end{gather*}

If $X$ and $Y$ are Asplund, then the \emph{fuzzy sum rule} (Lemma~\ref{l02}) is applicable to function $f$:
for any $\eps>0$,
\begin{gather*}\label{5f01}
\sd f(x,y)
\subset\bigcup_{\substack{
\|(x',y')-(x,y)\|<\eps,\
(x',y')\in\gph F\\
(x^*,y^*)\in N_{\gph F}(x',y')\\
\|y''-y\|<\eps,\,v^*\in\sd g(y'')}} \{x^*,y^*+v^*\}+\eps\B_{X^*\times Y^*}.
\end{gather*}
By definition \eqref{sds-f},
\begin{align*}
|\sd{f}|_{\rho}(x,y)
&\ge\lim_{\eps\downarrow0}\Biggl(\inf_{\substack{
\|(x',y')-(x,y)\|<\eps,\
(x',y')\in\gph F\\
(x^*,y^*)\in N_{\gph F}(x',y')\\
\|y''-y\|<\eps,\,v^*\in\sd g(y'')\\ \|y^*+v^*\|<\rho}} \|x^*\|-\eps\Biggr)
\\
&=\lim_{\eps\downarrow0}\inf_{\substack{
\|(x',y')-(x,y)\|<\eps,\
(x',y')\in\gph F\\
x^*\in D^*F(x',y')(y^*)\\
\|y''-y\|<\eps,\,v^*\in\sd g(y'')\\ \|y^*-v^*\|<\rho}} \|x^*\|
\\
&=\lim_{\eps\downarrow0}\inf_{\substack{
\|(x',y')-(x,y)\|<\eps,\
(x',y')\in\gph F\\
x^*\in D^*F(x',y')(\sd g(y'')+\rho\B^*)\\
\|y''-y\|<\eps}} \|x^*\|
\\
&=\lim_{\eps\downarrow0}\inf_{\substack{
\|(x',y')-(x,y)\|<\eps\\
(x',y')\in\gph F\\
\|y''-y\|<\eps}}\inf_{\substack{x^*\in D^*F(x',y')(\sd g(y'')+\rho\B^*)}} \|x^*\|
\\
&=\liminf_{\substack{ (x',y')\to(x,y),\,y''\to y\\(x',y')\in\gph F}}\; \inf_{\substack{x^*\in D^*F(x',y')(\sd g(y'')+\rho\B^*)}} \|x^*\|
\\
&=\liminf_{\substack{ (x',y')\to(x,y),\,y''\to y\\(x',y')\in\gph F}}\; \inf_{\substack{x^*\in D^*F(x',y')(q\|y''-\by\|^{q-1}J(y''-\by)+\rho\B^*)}} \|x^*\|
\\
&=\liminf_{\substack{ (x',y')\to(x,y),\,y''\to y\\(x',y')\in\gph F}}\; \inf_{\substack{x^*\in D^*F(x',y')(J(y''-\by)+\xi_q(y'')\rho\B^*)}} q\|y''-\by\|^{q-1}\|x^*\|
\\
&=q\|y-\by\|^{q-1}\liminf_{\substack{ (x',y')\to(x,y),\,y''\to y\\(x',y')\in\gph F}}\; \inf_{\substack{x^*\in D^*F(x',y')(J(y''-\by)+\xi_q(y'')\rho\B^*)}} \|x^*\|.
\end{align*}

(ii) The proof is similar to that of (i).
Instead of the fuzzy sum rule, one can use either the differentiable sum rule (part (ii) of Lemma~\ref{l02}) when $Y$ is Fr\'echet smooth, or the convex sum rule (part (iii) of Lemma~\ref{l02}) when $F$ is convex and $q=1$ to write down the representation:
\begin{gather*}\label{5f05}
\sd f(x,y)
=\bigcup_{(x^*,y^*)\in N_{\gph F}(x,y)} \set{x^*,q\|y-\by\|^{q-1}J(y-\by)+y^*},
\end{gather*}
By definition \eqref{sds-f},
\begin{align*}
|\sd{f}|_{\rho}(x,y)
&=\inf_{\substack{
(x^*,y^*)\in N_{\gph F}(x,y)\\
v^*\in J(y-\by),\, \|y^*+q\|y-\by\|^{q-1}v^*\|<\rho}} \|x^*\|
\\
&=q\|y-\by\|^{q-1}\inf_{\substack{
(x^*,y^*)\in N_{\gph F}(x,y)\\
v^*\in J(y-\by),\, \|y^*+v^*\|<\xi_q(y)\rho}} \|x^*\|
\\
&=q\|y-\by\|^{q-1}\inf_{\substack{
x^*\in D^*F(x,y)(J(y-\by)+\xi_q(y)\rho\B^*)}} \|x^*\|
\\
&=q\|y-\by\|^{q-1}\;|\sd{F}|_{\xi_q(y)\rho}(x,y).
\end{align*}
\qed\end{proof}

\subsection{Strict subdifferential slopes}

Using subdifferential $\rho$-slopes \eqref{6srs} and \eqref{6asrs}, we define now the \emph{strict subdifferential $q$-slope} and the \emph{approximate strict subdifferential $q$-slope} of $F$ at $(\bx,\by)$:
\begin{align}\label{q-sss}
\overline{|\sd{F}|}{}_q(\bar{x},\by):=&
q\lim_{\rho\downarrow0}
\inf_{\substack{\|x-\bx\|<\rho,\,\|y-\by\|<\rho\\
(x,y)\in\gph F,\,x\notin F\iv(\by)}}\,
\|y-\by\|^{q-1}\;|\sd{F}|_{\xi_q(y)\rho}(x,y),
\\\label{q-asss}
\overline{|\sd{F}|}{}^a_q(\bar{x},\by):=&
q\lim_{\rho\downarrow0}
\inf_{\substack{\|x-\bx\|<\rho,\,\|y-\by\|<\rho\\
(x,y)\in\gph F,\,x\notin F\iv(\by)}}\,
\|y-\by\|^{q-1}\;|\sd{F}|^a_{\xi_q(y)\rho}(x,y)
\end{align}
and their modified analogues:
\begin{multline}\label{q-msss}
\overline{|\sd{F}|}{}_q^{+}(\bar{x},\by):=\\
\lim_{\rho\downarrow0}
\inf_{\substack{\|x-\bx\|<\rho,\,\|y-\by\|<\rho\\
(x,y)\in\gph F,\,
x\notin F\iv(\by)}}
\max\left\{q\|y-\by\|^{q-1}|\sd{F}|_{\xi_q(y)\rho}(x,y), \frac{\|y-\by\|^{q}}{\|x-\bx\|}\right\},
\end{multline}
\begin{multline}\label{q-masss}
\overline{|\sd{F}|}{}_q^{a+}(\bar{x},\by):=\\
\lim_{\rho\downarrow0}
\inf_{\substack{\|x-\bx\|<\rho,\,\|y-\by\|<\rho\\
(x,y)\in\gph F,\,
x\notin F\iv(\by)}}
\max\left\{q\|y-\by\|^{q-1}|\sd{F}|^a_{\xi_q(y)\rho}(x,y), \frac{\|y-\by\|^{q}}{\|x-\bx\|}\right\}.
\end{multline}

With the help of Proposition~\ref{6P1}, we can formulate representations for the strict outer subdifferential slopes \eqref{ssds-f} and \eqref{mssds-f} in the case when $f$ is given by \eqref{f}.

\begin{proposition}\label{P5}
Let $f$ be given by \eqref{f}.
\begin{enumerate}
\item
If $X$ and $Y$ are Asplund and $\gph F$ is locally closed near $(\bx,\by)$, then
\begin{gather*}
\overline{|\sd{f}|}{}^>(\bar{x},\by) \ge\overline{|\sd{F}|}{}^a_q(\bar{x},\by)
\quad\mbox{and}\quad \overline{|\sd{f}|}{}^{>+}(\bar{x},\by) \ge\overline{|\sd{F}|}{}^{a+}_q(\bar{x},\by).
\end{gather*}
\item
If either
$F$ is convex near $(\bx,\by)$ and $q=1$ or $Y$ is Fr\'echet smooth, then
\begin{gather*}
\overline{|\sd{f}|}{}^>(\bar{x},\by) =\overline{|\sd{F}|}{}_q(\bar{x},\by) \quad\mbox{and}\quad \overline{|\sd{f}|}{}^{>+}(\bar{x},\by) =\overline{|\sd{F}|}{}^+_q(\bar{x},\by).
\end{gather*}
\end{enumerate}
\end{proposition}
\begin{proof}
(i) By \eqref{ssds-f} and Proposition~\ref{6P1}(i),
\begin{align*}
\overline{|\sd{f}|}{}^>(\bar{x},\by)
\ge&\lim_{\rho\downarrow0}
\inf_{\substack{\|(x,y)-(\bx,\by)\|<\rho\\
(x,y)\in\gph F,\,
x\notin F\iv(\by)}}
\\&
q\|y-\by\|^{q-1}
\lim_{\eps\downarrow0}
\inf_{\substack{\|(x',y')-(x,y)\|<\eps,\, \|y''-y\|<\eps\\x^*\in D^*F(x',y')(J(y''-\by)+\xi_q(y'')\rho\B^*)\\ (x',y')\in\gph F}}
\|x^*\|.
\end{align*}
For fixed $\rho\in(0,1)$ and $(x,y)$ with $x\notin F\iv(\by)$ and a sufficiently small $\eps>0$, it holds $B_\eps(x)\cap F\iv(\by)=\emptyset$, $\|(x,y)-(\bx,\by)\|+\eps<\rho$, and $\|y-\by\|^{q-1}\ge(1-\rho)\|y'-\by\|^{q-1}$ for all $y'\in B_\eps(y)$.
Besides, $\|y''-y'\|\le \|y''-y\|+\|y'-y\|<2\eps$.
Hence,
\begin{align*}
\overline{|\sd{f}|}{}^>(\bar{x},\by)
&\ge q\lim_{\rho\downarrow0}
\inf_{\substack{\|(x',y')-(\bx,\by)\|<\rho\\
(x',y')\in\gph F,\,
x'\notin F\iv(\by)}}
\\&
\qquad\qquad (1-\rho)\|y'-\by\|^{q-1}\lim_{\eps\downarrow0}
\inf_{\substack{\|y''-y'\|<2\eps\\x^*\in D^*F(x',y')(J(y''-\by)+\xi_q(y'')\rho\B^*)}}
\|x^*\|
\\&
=q\lim_{\rho\downarrow0}
\inf_{\substack{\|(x',y')-(\bx,\by)\|<\rho\\
(x',y')\in\gph F,\,
x'\notin F\iv(\by)}}\,\|y'-\by\|^{q-1}\, |\sd{F}|^a_{\xi_q(y')\rho}(x,y)
=\overline{|\sd{F}|}{}^a_q(\bar{x},\by).
\end{align*}
The proof of the other inequality goes along the same lines.

(ii) follows from Proposition~\ref{6P1}(ii) in view of definitions \eqref{ssds-f}, \eqref{mssds-f}, \eqref{q-sss}, and \eqref{q-msss}.
\qed\end{proof}

The next proposition provides some simplifications in the representations \eqref{q-sss}--\eqref{q-masss}.

\begin{proposition}\label{7P3}
The following assertions hold true:
\begin{enumerate}
\item
$\ds\overline{|\sd{F}|}{}_q(\bar{x},\by)\ge
q\lim_{\rho\downarrow0}
\inf_{\substack{\|x-\bx\|<\rho,\,\|y-\by\|<\rho\\
(x,y)\in\gph F,\,x\notin F\iv(\by)}}\,
\|y-\by\|^{q-1}\;|\sd{F}|_{\rho}(x,y)$;
\item
$\ds\overline{|\sd{F}|}{}^a_q(\bar{x},\by)\ge
q\lim_{\rho\downarrow0}
\inf_{\substack{\|x-\bx\|<\rho,\,\|y-\by\|<\rho\\
(x,y)\in\gph F,\,x\notin F\iv(\by)}}\,
\|y-\by\|^{q-1}\;|\sd{F}|^a_{\rho}(x,y)$;
\item
$\ds\overline{|\sd{F}|}{}_q^{+}(\bar{x},\by)\ge
\lim_{\rho\downarrow0}
\inf_{\substack{\|x-\bx\|<\rho,\,\|y-\by\|<\rho\\
(x,y)\in\gph F,\,
x\notin F\iv(\by)}}
\max\left\{q\|y-\by\|^{q-1}|\sd{F}|_{\rho}(x,y), \frac{\|y-\by\|^{q}}{\|x-\bx\|}\right\}$;
\item
$\ds\overline{|\sd{F}|}{}_q^{a+}(\bar{x},\by)\ge
\lim_{\rho\downarrow0}
\inf_{\substack{\|x-\bx\|<\rho,\,\|y-\by\|<\rho\\
(x,y)\in\gph F,\,
x\notin F\iv(\by)}}
\max\left\{q\|y-\by\|^{q-1}|\sd{F}|^a_{\rho}(x,y), \frac{\|y-\by\|^{q}}{\|x-\bx\|}\right\}$.
\end{enumerate}
If $q=1$, the above relations hold as equalities.
\end{proposition}

\begin{proof}
We consider the first inequality.
The others can be treated in the same way.
If $\overline{|\sd{F}|}{}_q(\bar{x},\by)=\infty$, the inequality holds trivially.
Let $\overline{|\sd{F}|}{}_q(\bar{x},\by)<\ga<\infty$.
Fix an arbitrary $\rho>0$ and choose an $\al>0$ and a $\rho'\in(0,\rho)$ such that $\al(\rho')^{1-q}<q$ and $\rho'<\al\rho$.
By \eqref{q-sss}, there exists an $(x,y)\in\gph F$ with $\|x-\bx\|<\rho'$, $\|y-\by\|<\rho'$ and $x\notin F\iv(\by)$; a $y^*\in Y^*$, an $x^*\in D^*F(x,y)(y^*)$, and a $v^*\in J(y-\by)$ such that $\|v^*-y^*\|\le\|y-\by\|^{1-q}\rho'/q$ and $q\|y-\by\|^{q-1}\|x^*\|<\ga$.
Hence, $\|x-\bx\|<\rho$, $\|y-\by\|<\rho$, and $\|v^*-y^*\|\le\al\iv\rho'<\rho$, and consequently the \RHS\ of (i) is less than $\ga$.
The conclusion follows since $\ga$ was chosen arbitrarily.

If $q=1$, then the \RHS s of (i) and \eqref{q-sss} coincide.
\qed\end{proof}

The next statement summarizes the relationships between the $q$-slopes.
It is a consequence of the definitions and Propositions~\ref{P3} and \ref{P5}.

\begin{proposition}[Relationships between slopes]\label{7P4}
\begin{enumerate}
\item
$\ds|\nabla{F}|_{q,\rho}^{\diamond}(x,y)\ge
\max\left\{q(d(y,\by))^{q-1}\;|\nabla{F}|_{\rho}(x,y), \frac{(d(y,\by))^q} {d_\rho((x,y),(\bx,\by))}\right\}$\\
for all $\rho>0$ and $(x,y)\in\gph F$ with $y\ne\by$;
\item
$\ds\overline{|\nabla{F}|}{}_q^{\diamond}(\bar{x},\by)\ge
\overline{|\nabla{F}|}{}^+_{q}(\bx,\by)\ge
\overline{|\nabla{F}|}{}_{q}(\bx,\by)$.
\cnta
\end{enumerate}
Suppose $X$ and $Y$ are normed spaces.
\begin{enumerate}
\cntb
\item
$\overline{|\sd{F}|}{}^a_q(\bar{x},\by) \le\overline{|\sd{F}|}{}_q(\bar{x},\by)\le \overline{|\sd{F}|}{}^+_q(\bar{x},\by)$ and\\
$\overline{|\sd{F}|}{}^a_q(\bar{x},\by) \le\overline{|\sd{F}|}{}^{a+}_q(\bar{x},\by) \le\overline{|\sd{F}|}{}^+_q(\bar{x},\by)$;
\item
$\overline{|\nabla{F}|}{}_q(\bar{x},\by)
\ge
\overline{|\sd{F}|}{}^{a}_q(\bar{x},\by)$ and $\overline{|\nabla{F}|}{}^+_q(\bar{x},\by)
\ge
\overline{|\sd{F}|}{}^{a+}_q(\bar{x},\by)$\\
provided that $X$ and $Y$ are Asplund and $\gph F$ is locally closed near $(\bar{x},\by)$;
\item
$\ds\overline{|\nabla{F}|}{}_q(\bar{x},\by) =\overline{|\sd{F}|}{}_q(\bar{x},\by)$ and $\ds\overline{|\nabla{F}|}{}^+_q(\bar{x},\by) =\overline{|\sd{F}|}{}^+_q(\bar{x},\by)$\\
provided that $Y$ is Fr\'echet smooth and one of the following conditions is satisfied:
\begin{enumerate}
\item
$X$ is Asplund and $\gph F$ is locally closed near $(\bar{x},\by)$;
\item
$F$ is convex near $(\bar{x},\by)$;
\end{enumerate}
\item
$\ds\overline{|\nabla{F}|}{}^{\diamond}_1(\bar{x},\by)=
\overline{|\nabla{F}|}{}^{+}_1(\bx,\by)=
\overline{|\nabla{F}|}{}_1(\bx,\by)= \overline{|\sd{F}|}{}^{+}_1(\bx,\by)= \overline{|\sd{F}|}{}_1(\bx,\by)$\\
provided that
$F$ is convex near $(\bar{x},\by)$.
\end{enumerate}
\end{proposition}

\subsection{Limiting outer $q$-coderivative}
In this subsection, $X$ and $Y$ are finite dimensional normed linear spaces.

In finite dimensions, strict subdifferential $q$-slopes \eqref{q-sss} and \eqref{q-asss} coincide and can be equivalently expressed in terms of the \emph{limiting outer $q$-coderivative} $\overline{D}{}^{*>}_q F(\bx,\by)$ of $F$ at $(\bx,\by)$ defined by its graph as follows:
\begin{align}\notag
\gph\overline{D}{}^{*>}_q F& (\bx,\by):= \Big\{(y^*,x^*)\in Y^*\times X^*\mid
\\\notag
&\exists (x_k,y_k,x^*_k,y^*_k,v^*_k)\subset X\times Y\times X^*\times Y^*\times Y^*\;\mbox{such that}
\\\notag
&(x_k,y_k)\in\gph F,\;x_k\notin F\iv(\by),
\\\notag
&(y_k^*,x_k^*)\in\gph{D}{}^{*}F(x_k,y_k),\;
v^*_k\in J(y_k-\by),
\\\notag
&(x_k,y_k)\to(\bx,\by),\; y^*_k-q\|y_k-\by\|^{q-1}v^*_k\to0,
\\\label{D*31}
&\|y^*\|x^*_k\to x^*,\;
\mbox{if}\; y^*\ne0,\; \mbox{then}\; \frac{y^*_k}{\|y_k^*\|}\to\frac{y^*}{\|y^*\|}\Big\}.
\end{align}
This set is a closed cone in $X\times Y$.
Hence, the limiting outer $q$-coderivative is a closed positively homogeneous \SVM.

\begin{proposition}\label{ows3}
$\ds\overline{|\sd{F}|}{}_q(\bar{x},\by) =\overline{|\sd{F}|}{}^{a}_q(\bar{x},\by) =\inf\limits_ {\substack{x^*\in\overline{D}{}^{*>}_q F(\bx,\by) (\Sp^*_{Y^*})}} \|x^*\|.$
\end{proposition}

\begin{proof}
We first prove that
\begin{gather}\label{ows3.1}
\ds\overline{|\sd{F}|}{}_q(\bar{x},\by) =\inf\limits_ {\substack{x^*\in\overline{D}{}^{*>}_q F(\bx,\by) (\Sp^*_{Y^*})}} \|x^*\|.
\end{gather}
Let $(y^*,x^*)\in\gph\overline{D}{}^{*>}_qF(\bx,\by)$, $\|y^*\|=1$, and $\rho>0$.
Choose an arbitrary sequence $(x_k,y_k,x^*_k,y^*_k,v^*_k,\al_k)$ corresponding to $(y^*,x^*)$ in accordance with definition~\eqref{D*31}.
Then, for a sufficiently large $k$, it holds $\|x_k-\bx\|<\rho$, $\|y_k-\by\|<\rho$, $(x_k,y_k)\in\gph F$, $x_k\notin F\iv(\by)$, $y_k^*\in q\|y_k-\by\|^{q-1}J(y_k-\by)+\rho\B^*$, $x_k^*\in D^*F(x_k,y_k)(y_k^*)$, and $\|x_k^*-x^*\|<\rho$.
Hence, by \eqref{6srs} and \eqref{q-sss}, $\overline{|\sd{F}|}{}_q(\bar{x},\by) \le\|x_k^*\|<\|x^*\|+\rho$ and consequently,
\begin{gather}\label{ows3.2}
\overline{|\sd{F}|}{}_q(\bar{x},\by) \le\inf\limits_ {\substack{x^*\in\overline{D}{}^{*>}_qF(\bx,\by) (\Sp^*_{Y^*})}} \|x^*\|.
\end{gather}

Conversely, by definitions \eqref{6srs} and \eqref{q-sss}, there exist sequences $(x_k,y_k)\to(\bx,\by)$ with $(x_k,y_k)\in\gph F$, $x_k\notin F\iv(\by)$ and $(x_k^*,y_k^*,v_k^*)\in X^*\times Y^*\times Y^*$ with $(y_k^*,x_k^*)\in\gph{D}{}^{*}F(x_k,y_k)$,  $v^*_k\in J(y_k-\by)$ such that $y^*_k-q\|y_k-\by\|^{q-1}v^*_k\to0$ and $\|x_k^*\|\to\overline{|\sd{F}|}{}_q(\bar{x},\by)$.
Passing to subsequences if necessary, we can assume that $x_k^*\to x^*\in X^*$ and either $y_k^*\ne0$ for all $k\in\N$, or $y_k^*=0$ for all $k\in\N$. In the first case, we can assume that $y_k^*/\|y_k^*\|\to y^*\in\Sp_{Y^*}^*$, and consequently, by definition \eqref{D*31}, $(x^*,y^*)\in\gph\overline{D}{}^{*>}_qF(\bx,\by)$.
In the second case, $(x^*,y^*)\in\gph\overline{D}{}^{*>}_qF(\bx,\by)$ for any $y^*\in Y^*$.
Hence,
$$
\overline{|\sd{F}|}{}_q(\bar{x},\by)=\|x^*\| \ge\inf\limits_ {\substack{x^*\in\overline{D}{}^{*>}_qF(\bx,\by) (\Sp^*_{Y^*})}} \|x^*\|.
$$
This together with \eqref{ows3.2} proves \eqref{ows3.1}.

The remaining equality
\begin{gather*}
\ds\overline{|\sd{F}|}{}^a_q(\bar{x},\by) =\inf\limits_ {\substack{x^*\in\overline{D}{}^{*>}_q F(\bx,\by) (\Sp^*_{Y^*})}} \|x^*\|
\end{gather*}
follows from comparing definitions \eqref{6srs} and \eqref{6asrs} thanks to the upper semicontinuity of the duality mapping.
\qed\end{proof}

\begin{remark}
The above definition of the limiting outer $q$-coderivative follows the original idea of limiting coderivatives; cf. \cite{Mor06.1}.
In particular, it defines a positively homogeneous \SVM\ with a not necessarily convex graph.
However, there are also several important distinctions.
Firstly, similar to the corresponding definition introduced in \cite{IofOut08}, this is an ``outer'' object: only sequences $(x_k,y_k)\in\gph F$ with $x_k$ components lying outside of the set $F\iv(\by)$ are taken into account.
Secondly, as it is reflected in its name, the limiting outer $q$-coderivative depends on $q$.
It is not excluded in the definition that $\|v_k^*\|\to\infty$ and consequently $\|y_k^*\|\to\infty$, and nevertheless the sequence $(y_k^*)$ produces a finite element $y^*\in Y$.
Note that restricting definition \eqref{D*31} to only ``outer'' points, while reasonable for characterizing the (H\"older) metric subregularity properties, prevents one from developing a reasonable calculus of such objects.
If this restriction in the definition is dropped, one can still use the resulting limiting coderivatives for formulating certain (stronger!) sufficient conditions.
\end{remark}

\begin{remark}
Analyzing the definition of the limiting outer $q$-coderivative and the proof of Proposition~\ref{ows3}, one can notice that there is no need to care much about the convergence of the sequences in $Y^*$.
The limiting outer $q$-coderivative in Proposition~\ref{ows3} can be replaced by the corresponding limiting set in $X^*$ only:
\begin{align}\notag
{S}{}^{*>}_qF& (\bx,\by):= \{x^*\in X^*\mid
\exists (x_k,y_k,x^*_k,y^*_k,v^*_k)\subset X\times Y\times X^*\times Y^*\times Y^*
\\\notag
&\mbox{such that}\;
(x_k,y_k)\in\gph F,\;x_k\notin F\iv(\by),\;
(y_k^*,x_k^*)\in\gph{D}{}^{*}F(x_k,y_k),\; \\\notag
&v^*_k\in J(y_k-\by),\;
(x_k,y_k)\to(\bx,\by),\; y^*_k-q\|y_k-\by\|^{q-1}v^*_k\to0,\; x^*_k\to x^*\}.
\end{align}
Proposition~\ref{ows3} remains true if $\overline{D}{}^{*>}_qF(\bx,\by)(\Sp^*_{Y^*})$ there is replaced by ${S}{}^{*>}_qF(\bx,\by)$.
This way one can also relax the assumption that $\dim Y<\infty$.
\end{remark}

\begin{remark}\label{Rem9}
One can define also a $q$-coderivative (indirect) counterpart of the modified strict subdifferential $q$-slopes \eqref{q-msss} and \eqref{q-masss}.
It is sufficient to add to the list of properties in definition \eqref{D*31} an additional requirement that $\|y_k-\by\|^q/\|x_k-\bx\|\to0$ as $k\to\infty$.
The corresponding set can be used for characterizing metric $q$-sub\-regularity.
However, the analogues of the equalities in Proposition~\ref{ows3} would not hold for it.
\end{remark}

\section{Criteria of H\"older metric subregularity}\label{Criteria}
In this section, if not specified otherwise, $X$ and $Y$ are metric spaces.
The next theorem is a consequence of Theorem~\ref{3T1}.
It is invariant on the choice of an admissible metric on $X\times Y$ (see Remark~\ref{rm2}).

\begin{theorem}\label{7T1}
\begin{enumerate}
\item
$\sr_q[F](\bx,\by)\le
\overline{|\nabla{F}|}{}^{\diamond}_q(\bar{x},\by)$;
\item
if $X$ and $Y$ are complete and $\gph F$ is locally closed near $(\bar{x},\by)$,
then
$\sr_q[F](\bx,\by)=
\overline{|\nabla{F}|}{}^{\diamond}_q(\bar{x},\by)$.
\end{enumerate}
\end{theorem}

The estimate in the next proposition can be useful when formulating necessary conditions of $q$-subregu\-larity.

\begin{proposition}\label{itu}
Suppose $X$ and $Y$ are normed spaces and $F$ is convex near $(\bx,\by)$.
Then, $q\cdot\sr_q[F](\bar{x},\by) \le\overline{|\sd{F}|}_q(\bar{x},\by)$. \end{proposition}

\begin{proof}
If $\sr_q[F](\bar{x},\by)=0$, the conclusion is trivial.
Suppose $0<\tau<\sr_q[F](\bar{x},\by)$ and $0<\ga<q$.
Then, there exists a $\rho\in(0,q-\ga)$ such that
\begin{gather}\label{iti1}
\tau d(x,F^{-1}(\by))<\|y-\by\|^q, \quad \forall x\in B_\rho(\bx)\setminus F\iv(\by),\; y\in F(x).
\end{gather}
Choose an arbitrary $(x,y)\in\gph F$ with $\|x-\bx\|<\rho$, $\|y-\by\|<\rho$, $x\notin F\iv(\by)$; $v^*\in J(y-\by)$; and $x^*\in D^*F(x,y)(v^*+\xi(y)\rho\B^*)$.
By \eqref{iti1}, one can find a point $u\in F\iv(\by)$ such that
\begin{gather}\label{iti4}
\tau\|x-u\|<\|y-\by\|^q.
\end{gather}
By the convexity of $F$, the Fr\'echet normal cone to its graph coincides with the normal cone in the sense of convex analysis and consequently, it holds
\begin{gather*}
\langle x^*,u-x\rangle\le\langle v^*,\by-y\rangle +\xi(y)\rho\|y-\by\|=-(1-\xi(y)\rho)\|y-\by\|.
\end{gather*}
Combining this with \eqref{iti4}, we have
\begin{align*}
q\|y-\by\|^{q-1}\|x^*\|\|u-x\| &\ge-q\|y-\by\|^{q-1}\langle x^*,u-x\rangle
\\
&\ge(q\|y-\by\|^{q-1}-\rho)\|y-\by\|
\\
&>q\|y-\by\|^q -(q-\ga)\|y-\by\|
\\
&\ge q\|y-\by\|^q -(q-\ga)\|y-\by\|^q
\\
&=\ga\|y-\by\|^q >\ga\tau\|u-x\|.
\end{align*}
Hence,
\begin{align*}
q\|y-\by\|^{q-1}\|x^*\|>\ga\tau,
\end{align*}
and it follows from definitions \eqref{q-sss} and \eqref{6srs} that $\overline{|\sd{F}|}_q(\bar{x},\by)\ge\ga\tau$.
Passing to the limit in the last inequality as $\ga\to q$ and $\tau\to\sr_q[F](\bar{x},\by)$, we arrive at the claimed inequality.
\qed\end{proof}

The next two corollaries summarize quantitative and qualitative criteria of H\"older metric subregularity of order $q$.

\begin{corollary}[Quantitative criteria]\label{7C1.1}
Let $\ga>0$.
Consider the following conditions:
\renewcommand {\theenumi} {\alph{enumi}}
\begin{enumerate}
\item
$F$ is H\"older metrically subregular of order $q$ at $(\bx,\by)$ with some $\tau>0$;
\item
$\overline{|\nabla{F}|}{}_q^{\diamond}(\bar{x},\by)>\ga$,\\ i.e., for some $\rho>0$ and any $(x,y)\in\gph F$ with $x\notin F\iv(\by)$, $d(x,\bx)<\rho$, and $d(y,\by)<\rho$, it holds $|\nabla{F}|_{q,\rho}^{\diamond}(x,y)>\ga$, and consequently there is a $(u,v)\in\gph F$ such that
\begin{gather*}\label{7ree}
(d(y,\by))^q-(d(v,\by))^q>\ga d_\rho((u,v),(x,y));
\end{gather*}
\item
$\ds\liminf_{\substack{x\to\bx\\x\notin F\iv(\by),\,y\in F(x)}} \frac{(d(y,\by))^q}{d(x,\bx)}>\ga$;
\item
$\overline{|\nabla{F}|}{}_q(\bar{x},\by)>\ga$,\\
i.e., for some $\rho>0$ and any $(x,y)\in\gph F$ with $x\notin F\iv(\by)$, $d(x,\bx)<\rho$, and $d(y,\by)<\rho$, it holds $q(d(y,\by))^{q-1}|\nabla{F}|_{\rho}(x,y)>\ga$, and consequently, for any $\eps>0$, there is a $(u,v)\in\gph F\cap B_\eps(x,y)$ such that
\begin{gather}\label{7phiree}
q(d(y,\by))^{q-1}(d(y,\by)-d(v,\by))>\ga d_\rho((u,v),(x,y));
\end{gather}
\item
$\overline{|\nabla{F}|}{}^+_q(\bar{x},\by)>\ga$,\\
i.e., for some $\rho>0$ and any $(x,y)\in X\times Y$ with $x\notin F\iv(\by)$, $d(x,\bx)<\rho$, $d(y,\by)<\rho$, and $(d(y,\by))^q/d(x,\bx)\le\ga$, it holds $q(d(y,\by))^{q-1}|\nabla{F}|_{\rho}(x,y)>\ga$ and consequently, for any $\eps>0$, there is a $(u,v)\in\gph F\cap B_\eps(x,y)$ such that \eqref{7phiree} holds true;
\sloppy
\item
$X$ and $Y$ are normed spaces and
$\overline{|\sd{F}|}{}_q^a(\bar{x},\by)>\ga$,\\
i.e.,
for some $\rho>0$ and any $(x,y)\in\gph F$ with $x\notin F\iv(\by)$, $\|x-\bx\|<\rho$, and $\|y-\by\|<\rho$, it holds
\begin{gather}\label{7f0}
q\|y-\by\|^{q-1}|\sd{F}|^a_{\xi(y)\rho}(x,y)>\ga
\end{gather}
and consequently, there exists an $\eps>0$ such that
\begin{gather}\label{7f}
\begin{aligned}
q\|y-\by\|^{q-1}&\|x^*\|>\ga\;\;\mbox{for all }\\ &x^*\in D^*F(x,y)(J(B_\eps(y-\by))+\xi(y)\rho\B^*);
\end{aligned}
\end{gather}
\item
$X$ and $Y$ are normed spaces and
$\overline{|\sd{F}|}{}_q^{a+}(\bar{x},\by)>\ga$,\\
i.e., for some $\rho>0$ and any $(x,y)\in X\times Y$ with $x\notin F\iv(\by)$, $\|x-\bx\|<\rho$, $\|y-\by\|<\rho$, and $\|y-\by\|^q/\|x-\bx\|\le\ga$, condition \eqref{7f0} holds and consequently, there exists an $\eps>0$ such that \eqref{7f} holds true;
\item
$X$ and $Y$ are normed spaces and
$\overline{|\sd{F}|}{}_q(\bar{x},\by)>\ga$,\\
i.e.,
for some $\rho>0$ and any $(x,y)\in\gph F$ with $x\notin F\iv(\by)$, $\|x-\bx\|<\rho$, and $\|y-\by\|<\rho$, it holds
\begin{gather}\label{7f20}
q\|y-\by\|^{q-1}|\sd{F}|_{\xi(y)\rho}(x,y)>\ga
\end{gather}
and consequently,
\begin{gather}\label{7f2}
\begin{aligned}
q\|y-\by\|^{q-1}&\|x^*\|>\ga\;\;\mbox{for all }\\ &x^*\in D^*F(x,y)(J(y-\by)+\xi(y)\rho\B^*);
\end{aligned}
\end{gather}
\item
$X$ and $Y$ are normed spaces and
$\overline{|\sd{F}|}{}^+_q(\bar{x},\by)>\ga$,\\
i.e., for some $\rho>0$ and any $(x,y)\in X\times Y$ with $x\notin F\iv(\by)$, $\|x-\bx\|<\rho$, $\|y-\by\|<\rho$, and $\|y-\by\|^q/\|x-\bx\|\le\ga$, condition
\eqref{7f20} holds, and consequently \eqref{7f2} holds true;
\item
$X$ and $Y$ are finite dimensional normed spaces and
\begin{gather*}
\|x^*\|>\ga\quad\mbox{for all } x^*\in \overline{D}{}^{*>}_q F(\bx,\by) (\Sp^*_{Y^*}).
\end{gather*}
\end{enumerate}
\renewcommand {\theenumi} {\roman{enumi}}
The following implications hold true:
\begin{enumerate}
\item
{\rm (c) \folgt (e)},
{\rm (d) \folgt (e)},
{\rm (e) \folgt (b)},
{\rm (f) \folgt (g) \folgt (i)}, {\rm (f)~\folgt (h)~\folgt (i)};
\item
if $\ga<\tau$, then {\rm (a) \folgt (b)};
\item
if $\tau\le\ga$, $X$ and $Y$ are complete, and $\gph F$ is locally closed near $(\bar{x},\by)$, then {\rm (b) \folgt (a)}.
\cnta
\end{enumerate}
Suppose $X$ and $Y$ are normed spaces.
\begin{enumerate}
\cntb
\item
If $F$ is convex, and $\ga<q\tau$, then {\rm (a) \folgt (h)};
\item
{\rm (f)~\folgt (d)} and {\rm (g)~\folgt (e)}\\
provided that $X$ and $Y$ are Asplund and $\gph F$ is locally closed near $(\bx,\by)$;
\item
{\rm (h)~\iff (d)} and {\rm (i)~\iff (e)}\\
provided that
$Y$ is Fr\'echet smooth and one of the following conditions is satisfied:
\begin{enumerate}
\item
$X$ is Asplund and $\gph F$ is locally closed near $(\bar{x},\by)$;
\item
$F$ is convex near $(\bar{x},\by)$;
\end{enumerate}
\item
{\rm (b)~\iff (d)~\iff (e)~\iff (h)~\iff (i)}\\
provided that
$F$ is convex near $(\bar{x},\by)$ and $q=1$;
\item
if $X$ and $Y$ are finite dimensional normed spaces, then {\rm (f)~\iff (h)~\iff (j)}.
\end{enumerate}
\end{corollary}

The conclusions of Corollary~\ref{7C1.1} are illustrated in
Fig.~\ref{fig.9}.

\begin{figure}[!htb]
$$\xymatrix{
&&&&{\rm (c)}\ar[d]
\\
&&{\rm (d)}\ar[rr]
\ar @{} [drr] |{\substack{X,Y\,{\rm Asplund}\\\gph F\, {\rm closed}}}
&&{\rm (e)}\ar[rr]
\ar@/_/@{-->}[ll]_{F\,{\rm convex},\,q=1}
\ar@/^1pc/@{-->}[dd]
&&{\rm (b)} \ar@/_/@{-->}[rr]_{\substack{\tau\le\ga\\X,Y\, {\rm complete}\\\gph F\, {\rm closed}}}
\ar@/_/@{-->}[ll]_{F\,{\rm convex},\,q=1}
&&{\rm (a)} \ar@/_/@{-->}[ll]_{\ga<\tau} \ar@/^5pc/@{-->}[lllllldd]|{\substack{X,Y\,{\rm normed}\\F\,{\rm convex}\\\ga<q\tau}}
\\
{\rm (j)}\ar@{-->}[rr]_(.65){\substack{\dim X<\infty\\\dim Y<\infty}}
\ar@{-->}[rrd]
&&{\rm (f)}\ar[rr]
\ar@{-->}[u]
\ar[d]
\ar@{-->}[ll]
&&{\rm (g)}
\ar@{-->}[u]
\ar[d]
\\
&&{\rm (h)}\ar[rr]\ar@{-->}[llu]\ar@/^/@{-->}[u]
&&{\rm (i)}
\ar@/_1pc/@{-->}[uu]_{\substack{[Y\,{\rm smooth}\smallskip\\(X\, {\rm Asplund}\\\gph F\, {\rm closed})\\{\rm or}\\ F\,{\rm convex}]\\{\rm or}\smallskip\\ F\,{\rm convex},\,q=1}}
\ar@/^/@{-->}[ll]^(.35){F\,{\rm convex},q=1}
}$$
\caption{Corollary~\ref{7C1.1} \label{fig.9}}
\end{figure}
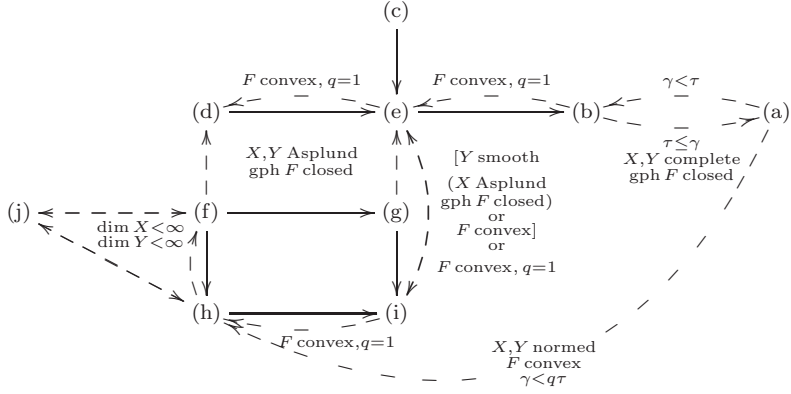

\begin{corollary}[Qualitative criteria]\label{C10}
Suppose $X$ and $Y$ are complete metric spaces and $\gph F$ is locally closed near $(\bar{x},\by)$.
Then,
$F$ is H\"older metrically subregular of order $q$ at $(\bx,\by)$ provided that one of the following conditions holds true:
\renewcommand {\theenumi} {\alph{enumi}}
\begin{enumerate}
\item
$\overline{|\nabla{F}|}{}^{\diamond}_q(\bar{x},\by)>0$; \item
$\ds\liminf_{\substack{x\to\bx\\x\notin F\iv(\by),\,y\in F(x)}} \frac{(d(y,\by))^q}{d(x,\bx)}>0$;
\item
$\overline{|\nabla{F}|}{}_q(\bar{x},\by)>0$;
\item
$\overline{|\nabla{F}|}{}^+_q(\bar{x},\by)>0$,
or equivalently,
$$\ds\lim_{\rho\downarrow0}
\inf_{\substack{d(x,\bx)<\rho,\,d(y,\by)<\rho,\,\frac{(d(y,\by))^q}{d(x,\bx)}<\rho\\
(x,y)\in\gph F,\,x\notin F\iv(\by)}}\,
(d(y,\by))^{q-1}|\nabla{F}|_{\rho}(x,y)>0.$$
\cnta
\end{enumerate}
If $X$ and $Y$ are Asplund spaces, then the following conditions are also sufficient:
\begin{enumerate}
\cntb
\item
$\overline{|\sd{F}|}{}^a_q(\bar{x},\by)>0$;
\item
$\overline{|\sd{F}|}{}^{a+}_q(\bar{x},\by)>0$, or equivalently,
$$\ds\lim_{\rho\downarrow0}
\inf_{\substack{\|x-\bx\|<\rho,\,\|y-\by\|<\rho,\,\frac{\|y-\by\|^q}{\|x-\bx\|}<\rho\\
(x,y)\in\gph F,\,x\notin F\iv(\by)}}\,
\|y-\by\|^{q-1}|\sd{F}|_{\xi(y)\rho}^a(x,y)>0.$$
\cnta
\end{enumerate}
The next two conditions:
\begin{enumerate}
\cntb
\item
$\overline{|\sd{F}|}{}_q(\bar{x},\by)>0$;
\item
$\overline{|\sd{F}|}{}^+_q(\bar{x},\by)>0$, or equivalently,
\begin{gather*}\label{exc3}
\ds\lim_{\rho\downarrow0}
\inf_{\substack{\|x-\bx\|<\rho,\,\|y-\by\|<\rho,\,\frac{\|y-\by\|^q}{\|x-\bx\|}<\rho\\
(x,y)\in\gph F,\,x\notin F\iv(\by)}}\,
\|y-\by\|^{q-1}|\sd{F}|_{\xi(y)\rho}(x,y)>0,
\end{gather*}
\cnta
\end{enumerate}
are sufficient provided that $X$ and $Y$ are Banach spaces and one of the following conditions is satisfied:
\begin{itemize}
\item
$X$ is Asplund and $Y$ is Fr\'echet smooth,
\item
$F$ is convex near $(\bar{x},\by)$ and either $Y$ is Fr\'echet smooth or $q=1$.
\end{itemize}
If $X$ and $Y$ are finite dimensional normed spaces, then the following condition is also sufficient:
\begin{enumerate}
\cntb
\item
$0\notin \overline{D}{}^{*>}_q F(\bx,\by) (\Sp^*_{Y^*})$.
\end{enumerate}
\renewcommand {\theenumi} {\roman{enumi}}
Moreover,
\begin{enumerate}
\item
condition {\rm (a)} is also necessary for the metric $q$-subregularity of $F$ at $(\bx,\by)$;
\item
{\rm (b) \folgt (d)},
{\rm (c) \folgt (d)},
{\rm (d) \folgt (a)},
{\rm (e) \folgt (f)},
{\rm (g) \folgt (h)}.
\cnta
\end{enumerate}
Suppose $X$ and $Y$ are Banach spaces.
\begin{enumerate}
\cntb
\item
If $X$ and $Y$ are Asplund, then {\rm (e) \folgt (c)} and {\rm (f) \folgt (d)};
\item
if $Y$ is Fr\'echet smooth and either $X$ is Asplund or $F$ is convex near $(\bar{x},\by)$, then {\rm (e) \iff (c)} and {\rm (f) \iff (d)};
\item
if $F$ is convex near $(\bx,\by)$, then
condition {\rm (g)} is also necessary for the metric $q$-sub\-regularity of $F$ at $(\bx,\by)$;
\item
if $F$ is convex near $(\bx,\by)$ and $q=1$, then {\rm (a) \iff (c) \iff (d) \iff \rm (g) \iff (h)};
\item
if $X$ and $Y$ are finite dimensional normed spaces, then
{\rm (e) \iff (g) \iff (i)}.
\end{enumerate}
\end{corollary}

The conclusions of Corollary~\ref{C10} are illustrated in
Fig.~\ref{fig.10}.

\begin{figure}[!htb]
$$\xymatrix@C=1cm{
&&*+[F]{\sr_q[F](\bx,\by)>0} \ar[d]
\\
&*+[F]{\ds\liminf_{\substack{x\to\bx\\x\notin F\iv(\by)\\ y\in F(x)}} \frac{(d(y,\by))^q}{d(x,\bx)}>0}\ar[dr]
&*+[F]{\overline{|\nabla{F}|}{}^{\diamond}_q(\bar{x},\by)>0}
\ar[u]
\ar@/^/@{-->}[d]^{F\,{\rm convex},\,q=1}
\ar
@{-->}`r /30pt[d] `^rd/4pt[dd] `_l[l] `[l]|{\substack{X,Y\,{\rm Banach}\\F\,{\rm convex}}} [lddd]
\\
&*+[F]{\overline{|\nabla{F}|}_q(\bar{x},\by)>0}
\ar[r]
\ar@/^/@{-->}[d]
\ar@{} [dr] |{\substack{X\,{\rm Banach},\,Y\,{\rm smooth}\smallskip\\X\,{\rm Asplund}\,{\rm or}\,F\,{\rm convex}}}
&*+[F]{\overline{|\nabla{F}|}{}^{+}_q(\bar{x},\by)>0}
\ar[u]
\ar@/_/@{-->}[d]
\ar@/_/@{-->}[l]_{F\,{\rm convex},\,q=1}
\\
*+[F]{0\notin \overline{D}{}^{*>}_q F(\bx,\by) (\Sp^*_{Y^*})}
\ar@{-->}[r]
\ar@{-->} [dr]^(.65){\substack{\dim X<\infty\\\dim Y<\infty}} &*+[F]{\overline{|\sd{F}|}{}^a_q(\bar{x},\by)>0}
\ar[r]
\ar@{-->}[l]
\ar@/^/@{-->}[u]^{X,Y\,{\rm Asplund}}
\ar[d]
&*+[F]{\overline{|\sd{F}|}{}^{a+}_q(\bar{x},\by)>0}
\ar@/_/@{-->}[u]_{X,Y\,{\rm Asplund}}
\ar[d]
\\
&*+[F]{\overline{|\sd{F}|}_q(\bar{x},\by)>0}
\ar[r]
\ar@{-->}@/^/[u]
\ar@{-->}[ul]
&*+[F]{\overline{|\sd{F}|}{}^{+}_q(\bar{x},\by)>0}
\ar@/_4.5pc/@{-->}[uu]|{\substack{X,Y\,{\rm Banach}\\F\,{\rm convex}\\q=1}}
\ar@/^/@{-->}[l]^{F\,{\rm convex},\,q=1}
}$$
\caption{Corollary~\ref{C10}}\label{fig.10}
\end{figure}
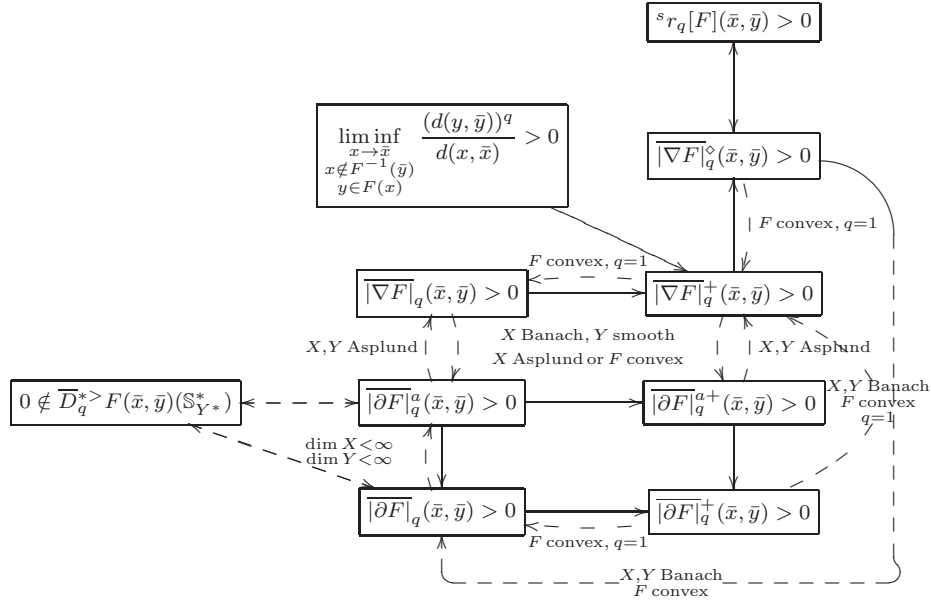

Another three subdifferential criteria of H\"older metric subregularity of order $q$ have been established recently by Li and Mordukhovich \cite{LiMor12} (cf. \cite{NgaiTronThe3}).
In the case when $Y$ is a Banach space, they introduced two modifications of the duality mapping $J$ defined by \eqref{J}, namely, the \emph{$q$-dua\-lity mapping} $J^q$ and, given an $\eps\ge0$, its \emph{normalized $\eps$-en\-largement} $J^q_\eps$ (both acting $Y\setminus\{0\}\rightrightarrows Y^*$):
\begin{gather}\notag
J^q(y):=q\|y\|^{q-1}J(y),\\\label{DuMap}
J^q_\eps(y):=\left\{\left.\frac{y^*+\eps v^*}{\|y^*+\eps v^*\|}\right|\, y^*\in J^q(y),\,\|v^*\|\le1,\,y^*+\eps v^*\ne0\right\}.
\end{gather}
Using \eqref{DuMap}, the authors defined for the mapping $F$ two nonnegative constants:
\begin{align}\notag
\al:=\sup_{\eps>0}\inf\bigl\{ &q\|x^*\|\cdot\|y'-\by\|^{q-1} \mid
(x,y)\in\gph F,\, x\notin F\iv(\by),\\\notag
&\|x-\bx\|<\eps,\, \|y-\by\|<\min\{\eps,\|x-\bx\|^\frac{1}{2}\},\\\label{alp}
&\|y'-y\|<\|x-\bx\|^\frac{1}{q},\,y'\ne\by,\, x^*\in D^*F(x,y)(J^q_\eps(y'-\by))\bigr\},
\\\notag
\be:=\sup_{\eps>0}\inf\bigl\{ &q\|x^*\|\cdot\|y-\by\|^{q-1} \mid
(x,y)\in\gph F,\, x\notin F\iv(\by),\\\notag
&\|x-\bx\|<\eps,\, \|y-\by\|<\min\{\eps,\|x-\bx\|^\frac{1}{2}\},\\\label{bet}
&x^*\in D^*F(x,y)(J^q_\eps(y-\by))\bigr\},
\end{align}
which played a crucial role when determining the H\"older metric subregularity of $F$.

Using the notation adopted in the current article and with $\al$ and $\be$ defined by \eqref{alp} and \eqref{bet}, respectively, \cite[Theorems~3.3, 5.1, and 5.3]{LiMor12} can be formulated in the following way.

\begin{theorem}\label{LM}
{\rm (i)}
Suppose $X$ and $Y$ are Asplund and $\gph F$ is locally closed near $(\bx,\by)$.
Then, 1)~$\al\le\sr_q[F](\bx,\by)$ and 2)~condition $\al>0$ is sufficient for H\"older metric subregularity of order $q$ of $F$ at $(\bx,\by)$.

{\rm (ii)}
Suppose $X$ is Asplund, $Y$ is Fr\'echet smooth, and $\gph F$ is locally closed near $(\bx,\by)$.
Then, 1)~$\be\le\sr_q[F](\bx,\by)$ and 2)~condition $\be>0$ is sufficient for H\"older metric subregularity of order $q$ of $F$ at $(\bx,\by)$.

{\rm (iii)}
Suppose $X$ is Banach, $Y$ is Fr\'echet smooth, and $\gph F$ is closed and convex.
Then, condition $\be>0$ is necessary and sufficient for H\"older metric subregularity of order $q$ of $F$ at $(\bx,\by)$.
\end{theorem}

The assertions of the above theorem are obvious consequences of Proposition~\ref{7P4} and Corollary~\ref{7C1.1} thanks to the next proposition.

\begin{proposition}\label{LM2}
Suppose $X$ and $Y$ are normed spaces.
Then,
\begin{enumerate}
\item
$\al\leq \overline{|\sd{F}|}{}^{+a}_q(\bar{x},\by)$,
\item
$\overline{|\sd{F}|}{}_q(\bar{x},\by) \le\be\leq\overline{|\sd{F}|}{}^{+}_q(\bar{x},\by)$.
\end{enumerate}
\end{proposition}

\begin{proof}
(i) If $\overline{|\sd{F}|}{}^{+a}_q(\bar{x},\by)=\infty$, the inequality is satisfied trivially.
Suppose
$\overline{|\sd{F}|}{}^{+a}_q(\bar{x},\by)< \ga<\infty$.
Then, by definition \eqref{q-masss},
$$\overline{|\sd{F}|}{}_q^{+a}(\bar{x},\by)\ge \lim_{\rho\downarrow0}
\inf_{\substack{\|x-\bx\|<\rho,\,\|y-\by\|<\rho\\
\|y-\by\|^q/\|x-\bx\|<\ga\\
(x,y)\in\gph F,\,x\notin F\iv(\by)}}
q\|y-\by\|^{q-1} |\sd{F}|_{\xi(y)\rho}^a(x,y).$$
If $0<\rho< \ga^{-\frac{1}{1-q/2}}$, $0<\|x-\bx\|<\rho$, and $\|y-\by\|^q/\|x-\bx\|<\ga$, then,
$$\|y-\by\|<(\ga\|x-\bx\|)^{\frac{1}{q}}= \|x-\bx\|^{\frac{1}{2}} \left(\ga\|x-\bx\|^{1-\frac{q}{2}}\right)^{\frac{1}{q}}< \|x-\bx\|^{\frac{1}{2}}.$$
Hence,
\begin{gather}\label{LM2-1}
\overline{|\sd{F}|}{}_q^{+a}(\bar{x},\by)\ge \lim_{\rho\downarrow0}
\inf_{\substack{\|x-\bx\|<\rho,\,\|y-\by\|<\rho\\
\|y-\by\|<\|x-\bx\|^{\frac{1}{2}}\\
(x,y)\in\gph F,\,x\notin F\iv(\by)}}
q\|y-\by\|^{q-1} |\sd{F}|_{\xi(y)\rho}^a(x,y).
\end{gather}
If $(x,y)\in\gph F$, $y\ne\by$, and $\eps\in(0,\|y-\by\|)$, then, by definition \eqref{6asrs}, it holds
\begin{align*}
|\sd{F}|_{\xi(y)\rho}^a(x,y)&\ge \inf_{\substack{\|v-(y-\by)\|<\eps\\ x^*\in D^*F(x,y)(J(v)+\xi(y)\rho\B^*)}}
\|x^*\|
\\&=
\inf_{\substack{\|y'-y\|<\eps\\ x^*\in D^*F(x,y)(J(y'-\by)+\xi(y)\rho\B^*)}}
\|x^*\|
\\&=
\inf_{\substack{\|y'-y\|<\eps\\ x^*\in D^*F(x,y)(y^*+\xi(y)\rho v^*)\\
y^*\in J(y'-\by),\,\|v^*\|\le1}}
\|x^*\|.
\end{align*}
Let $\rho\in(0,1)$, $(x,y)\in\gph F$, $x\ne\bx$, and $y\ne\by$ be such that $\|y-\by\|^{1-q}\le q$.
Choose an $\eps\in(0,\|y-\by\|)$ such that inequality $\|y'-y\|<\eps$ implies the following estimates:
\begin{gather}\label{LM2-2}
\|y'-y\|<\|x-\bx\|^{\frac{1}{q}},\quad 1-\rho<(\|y'-\by\|/\|y-\by\|)^{q-1}<2.
\end{gather}
Then, given any $y'\in Y$ with $\|y'-y\|<\eps$, $y^*\in J(y'-\by)$, and
$\|v^*\|\le1$, we have
\begin{gather*}
\|y^*+\xi(y)\rho v^*\|>1-\rho>0,
\\
\rho':=\rho(\|y'-\by\|/\|y-\by\|)^{q-1}<2\rho, \\
y^*+\xi(y)\rho v^*=
y^*+\frac{\|y-\by\|^{1-q}}{q}\rho v^*= \frac{\|y'-\by\|^{1-q}}{q} (q\|y'-\by\|^{q-1}y^*+\rho' v^*),
\\
\frac{y^*+\xi(y)\rho v^*}{\|y^*+\xi(y)\rho v^*\|}=\frac{q\|y'-\by\|^{q-1}y^*+\rho' v^*}{\|q\|y'-\by\|^{q-1}y^*+\rho' v^*\|}\in J_{\rho'}^q(y'-\by)\subset J_{2\rho}^q(y'-\by).
\end{gather*}
Hence,
\begin{align*}
D^*F(x,y)(y^*+\xi(y)\rho v^*)&= \|y^*+\xi(y)\rho v^*\|D^*F(x,y) \left(\frac{y^*+\xi(y)\rho v^*}{\|y^*+\xi(y)\rho v^*\|}\right)
\\&\subset
\|y^*+\xi(y)\rho v^*\|D^*F(x,y) \left(J_{2\rho}^q(y'-\by)\right),
\end{align*}
and consequently,
\begin{align*}
|\sd{F}|_{\xi(y)\rho}^a(x,y)&\ge(1-\rho) \inf_{\substack{\|y'-y\|<\|x-\bx\|^{\frac{1}{q}},\, y'\ne\by\\ x^*\in D^*F(x,y)\left(J_{2\rho}^q(y'-\by)\right)}}
\|x^*\|.
\end{align*}
Combining this estimate with \eqref{LM2-1} and making use of \eqref{LM2-2}, we obtain
\begin{align*}
\overline{|\sd{F}|}{}^{+a}_q(\bar{x},\by)
&\ge \lim_{\rho\downarrow0}(1-\rho)^2
&\inf_{\substack{
\|x-\bx\|<\rho,\,\|y-\by\|<\rho\\
\|y-\by\|<\|x-\bx\|^{\frac{1}{2}}\\
(x,y)\in\gph F,\,x\notin F\iv(\by)}}\,
&q\|y'-\by\|^{q-1} 
\\&&&
\inf_{\substack{\|y'-y\|<\|x-\bx\|^{\frac{1}{q}},\, y'\ne\by\\ x^*\in D^*F(x,y)\left(J_{2\rho}^q(y'-\by)\right)}}
\|x^*\|\\
&\ge \lim_{\rho\downarrow0}
&\inf_{\substack{
\|x-\bx\|<2\rho,\,\|y-\by\|<2\rho\\
\|y-\by\|<\|x-\bx\|^{\frac{1}{2}}\\
(x,y)\in\gph F,\,x\notin F\iv(\by)}}\,
&q\|y'-\by\|^{q-1} 
\\&&&
\inf_{\substack{\|y'-y\|<\|x-\bx\|^{\frac{1}{q}},\, y'\ne\by\\ x^*\in D^*F(x,y)\left(J_{2\rho}^q(y'-\by)\right)}}
\|x^*\|=\al.
\end{align*}

(ii) As in the proof of (i),
we have from definition \eqref{q-msss}
\begin{align*}\label{LM2-3}
\overline{|\sd{F}|}{}_q^{+}(\bar{x},\by)&\ge \lim_{\rho\downarrow0}
\inf_{\substack{\|x-\bx\|<\rho,\,\|y-\by\|<\rho\\
\|y-\by\|<\|x-\bx\|^{\frac{1}{2}}\\
(x,y)\in\gph F,\,x\notin F\iv(\by)}}
q\|y-\by\|^{q-1} |\sd{F}|_{\xi(y)\rho}(x,y),
\end{align*}
where, by definition \eqref{6srs} and
using again (simplified versions of) the same arguments as in the proof of (i),
\begin{align*}
|\sd{F}|_{\xi(y)\rho}(x,y)= \inf_{\substack{x^*\in D^*F(x,y)(y^*+\xi(y)\rho v^*)\\
y^*\in J(y-\by),\,\|v^*\|\le1}}
\|x^*\|\ge(1-\rho) \inf_{\substack{x^*\in D^*F(x,y)\left(J_{\rho}^q(y-\by)\right)}}
\|x^*\|.
\end{align*}
The second inequality in (ii) follows from the last two estimates and definition \eqref{bet}.

Let $0<\eps<q^{\frac{1}{2-q}}$, $(x,y)\in\gph F$, $0<\|y-\by\|<\eps$, and $w^*\in J^q_\eps(y-\by)$, that is,
$$
w^*=\frac{q\|y-\by\|^{q-1}y^*+\eps v^*} {\bigl\|q\|y-\by\|^{q-1}y^*+\eps v^*\bigr\|}
$$
for some $y^*\in J(y-\by)$ and $v^*\in\B^*$. Observe that $\xi(y)=\|y-\by\|^{1-q}/q<\eps^{1-q}/q$,
\begin{gather*}
\bigl\|q\|y-\by\|^{q-1}y^*+\eps v^*\bigr\|\ge q\eps^{q-1}-\eps=\eps^{q-1}(q-\eps^{2-q})>0,
\\
q\|y-\by\|^{q-1}y^*+\eps v^*= q\|y-\by\|^{q-1}(y^*+\eps\xi(y) v^*).
\end{gather*}
Hence,
\begin{align*}
\|w^*-y^*\| &=\left\|\frac{y^*+\eps\xi(y) v^*} {\|y^*+\eps\xi(y)v^*\|}-y^*\right\|
=\frac{\bigl\|y^*(1-\|y^*+\eps\xi(y)v^*\|) +\eps\xi(y)v^*\bigr\|} {\|y^*+\eps\xi(y)v^*\|}
\\
&\le\frac{\bigl|1-\|y^*+\eps\xi(y)v^*\|\bigr| +\eps\xi(y)} {1-\eps\xi(y)}\le \frac{2\eps\xi(y)} {1-\eps\xi(y)}
\le\frac{2\eps} {1-\eps^{2-q}/q}\xi(y).
\end{align*}
Given a $\rho>0$, one can always choose an $\eps>0$ such that $2\eps/(1-\eps^{2-q}/q)<\rho$, and consequently $D^*F(x,y)(w^*)\subset D^*F(x,y)(y^*+\xi(y)\rho\B^*)$.
The first inequality follows from definitions \eqref{q-sss}, \eqref{6srs}, and \eqref{bet}.
\qed\end{proof}

\begin{proof}\emph{of Theorem~\ref{LM}}
The conclusions follow from Theorem~\ref{7T1}(ii) thanks to Propositions~\ref{7P4} and \ref{LM2}.

(i) requires parts (ii) and (iv) of Proposition~\ref{7P4} and part (i) of Proposition~\ref{LM2}.

(ii) requires parts (ii) and (v)(a) of Proposition~\ref{7P4} and part (ii) of Proposition~\ref{LM2}.

(iii) requires parts (ii) and (v)(b) of Proposition~\ref{7P4}, part (ii) of Proposition~\ref{LM2}, and, additionally, Proposition~\ref{itu}.
\qed\end{proof}

\begin{remark}
Assertions (i) and (ii) of Theorem~\ref{LM} are in general weaker than the corresponding ones in Corollary~\ref{7C1.1}.
They can be strengthened if inequality $\|y-\by\|<\min\{\eps,\|x-\bx\|^\frac{1}{2}\}$ in definitions \eqref{alp} and \eqref{bet} is replaced by a more restrictive (when $\|x-\bx\|<1$) one: $\|y-\by\|<\min\{\eps,\|x-\bx\|^\frac{1}{p}\}$ where $p$ can be any number in the interval $(q,2)$.
Proposition~\ref{LM2} remains true in this situation.
The only change required in its proof is replacing inequality $\rho<\ga^{-\frac{1}{1-q/2}}$ by the following one: $\rho<\ga^{-\frac{1}{1-q/p}}$.
\end{remark}

The next example illustrates the computation of the constants involved in the definition and characterizations of H\"older metric subregularity.

\begin{example}
Consider a mapping $F:\R\to\R$ given by
$$
F(x):=(x_+)^2=
\begin{cases}
x^2 & \text{if } x\ge0,\\
0 & \text{otherwise},
\end{cases}
$$
cf. \cite[Example~3.8]{LiMor12}.
It is obviously H\"older metrically subregular of order $q=1/2$ at $(0,0)$.
Note that $F\iv(0)=(-\infty,0]$ and, if $x>0$, then $d(x,F\iv(0))=x$ and $d(0,F(x))=x^2$.
This allows us to compute the modulus of H\"older metric subregularity \eqref{7CMR-g}:
\begin{gather*}
\sr_q[F](0,0)=
\liminf_{x\downarrow0} \frac{(x^2)^\frac{1}{2}}{x}=1.
\end{gather*}

This result can also be deduced from Theorem~\ref{7T1}(ii).
For that, one needs to compute the uniform strict $q$-slope \eqref{7uss}.
Let $x>0$, $y=x^2$, and $\rho>0$.
Then the nonlocal $(q,\rho)$-slope \eqref{7nls} of $F$ at $(x,y)$ takes the following form:
\begin{gather}\label{8nls}
|\nabla{F}|_{q,\rho}^{\diamond}(x,y)=
\sup_{u\ne x}
\frac{[x-u_+]_+}{\max\{|u-x|,\rho|u_+^2-x^2|\}}.
\end{gather}
If $u>x$, then $x-u_+=x-u<0$ and the expression under $\sup$ in the \RHS\ of \eqref{8nls} equals 0.
If $u\le0$, it equals $x/\max\{|u-x|,\rho x^2\}$ and is continuous and increasing on $(-\infty,0]$ as a function of $u$, attaining its maximum at $0$.
Hence,
\begin{gather*}
|\nabla{F}|_{q,\rho}^{\diamond}(x,y)=
\sup_{0\le u<x}
\frac{x-u}{\max\{x-u,\rho(x^2-u^2)\}}=
\frac{1}{\max\{1,\rho x\}},
\end{gather*}
and consequently $|\nabla{F}|_{q,\rho}^{\diamond}(x,y)=1$ if $\rho x\le1$.
It follows immediately from definition \eqref{7uss} that
$\overline{|\nabla{F}|}{}^{\diamond}_q(0,0)=1$.

H\"older metric subregularity of $F$ can also be established from the estimates in Proposition~\ref{7P4} after computing any of the local strict $q$-slopes \eqref{q-ss}, \eqref{aq-ss}, \eqref{q-sss}--\eqref{q-masss} which, in turn, depend on local $\rho$-slopes \eqref{srho}, \eqref{6srs}, and \eqref{6asrs}.
Observe that the last three constants do not depend on $q$.

Similarly to the above, for $x>0$, $y=x^2$, and $\rho>0$, one has
\begin{align*}
|\nabla{F}|_{\rho}(x,y)&=
\limsup_{u\to x,\,u\ne x}
\frac{[x^2-u^2]_+} {\max\{|u-x|,\rho|u^2-x^2|\}}\\ &=\lim_{u\uparrow x}
\frac{x+u} {\max\{1,\rho(x+u)\}} =\frac{2x} {\max\{1,2\rho x\}}
\end{align*}
and consequently, $|\nabla{F}|_{\rho}(x,y)=2x$ if $\rho x\le1/2$.
Mapping $F$ is differentiable and $D^*F(x,y)(y^*)=\{2xy^*\}$ for any $x\ge0$, $y=x^2$, and $y^*\in\R$.
Duality mapping \eqref{J} in $\R$ has a simple representation: $J(v)=1$ if $v>0$ and $J(v)=-1$ if $v<0$.
Thus,
\begin{gather*}
|\sd{F}|_{\rho}(x,y)=|\sd{F}|^a_{\rho}(x,y)
=\inf_{\substack{x^*\in D^*F(x,y)(1+\rho\B^*)}}
|x^*|=2x(1-\rho).
\end{gather*}

Observing that $(d(y,0))^{q-1}=1/x$ and $\xi(y)=2x$ in \eqref{q-sss}--\eqref{q-masss}, we arrive at
\begin{align*}
\overline{|\nabla{F}|}{}_q(\bar{x},\by) &=\overline{|\nabla{F}|}{}^+_q(\bar{x},\by)\\
&=\overline{|\sd{F}|}{}_q(\bar{x},\by) =\overline{|\sd{F}|}{}^+_q(\bar{x},\by) =\overline{|\sd{F}|}{}^a_q(\bar{x},\by) =\overline{|\sd{F}|}{}^{+a}_q(\bar{x},\by) =1.
\end{align*}
By Proposition~\ref{7P4}, this ensures that $F$ is H\"older metrically subregular of order $\frac{1}{2}$ at $(0,0)$ with modulus not less than 1.
\qedtr
\end{example}


\begin{acknowledgements}
The author wishes to thank the two anonymous referees for the careful reading of the manuscript and helpful comments and suggestions.
\end{acknowledgements}

\bibliographystyle{spmpsci}
\bibliography{buch-kr,kruger,kr-tmp}
\end{document}